\documentclass[11pt]{article}
\usepackage{fullpage}
\usepackage{hyperref,graphicx}
\usepackage{authblk}
\usepackage{amsmath}
\usepackage{amsfonts}
\usepackage{amsthm}
\usepackage{amssymb}
\usepackage{MnSymbol}
\usepackage{bbm}
\usepackage{listings}
\usepackage{enumerate}
\usepackage{cite}
\usepackage{etoolbox}

\newcommand{\R}{\mathbb{R}}

\newcommand{\Q}{\mathbb{Q}}
\newcommand{\E}{\mathbb{E}}

\renewcommand{\P}{\mathbb{P}}

\newcommand{\1}[1]{\mathbbm{1}_{\{#1\}}}

\newcommand{\calG}{\mathcal{G}}

\newcommand{\calA}{\mathcal{A}}

\newcommand{\normal}{\mathcal{N}}

\newcommand{\pdiffII}[3]{\ifstrequal{#2}{#3}
{\frac{\partial^2 #1}{\partial #2^2}}
{\frac{\partial^2 #1}{\partial #2 \partial #3}}
}
\newcommand{\diffII}[3]{\ifthenelse{\equal{#2}{#3}}
{\frac{d^2 #1}{d #2^2}}
{\frac{d^2 #1}{d #2 d #3}}
}
\newcommand{\diff}[2]{\frac{d #1}{d #2}}

\let\Pr\relax
\DeclareMathOperator{\Pr}{Pr}

\DeclareMathOperator{\Var}{Var}
\DeclareMathOperator{\Cov}{Cov}

\DeclareMathOperator{\tr}{tr}

\DeclareMathOperator{\Binom}{Binom}
\DeclareMathOperator{\Pois}{Pois}

\newcommand{\toD}{\stackrel{d}{\to}}

\newtheorem{theorem}{Theorem}[section]

\newtheorem{lemma}[theorem]{Lemma}

\newtheorem{proposition}[theorem]{Proposition}

\newtheorem{definition}[theorem]{Definition}

\newtheorem{question}[theorem]{Question}

\newtheoremstyle{example}{\topsep}{\topsep}%
     {}
     {}
     {\bfseries}
     {}
     {\newline}
     {\thmname{#1}\thmnumber{ #2}\thmnote{ #3}}

\theoremstyle{example}

\let\vec\relax
\DeclareMathOperator{\vec}{vec}
\DeclareMathOperator{\Multinom}{Multinom}
\DeclareMathOperator{\Normal}{\mathcal{N}}
\DeclareMathOperator{\diag}{diag}
\newcommand{\prob}[1]{\mathbb{P}_n\left(#1\right)}
\newcommand{\probcond}[1]{\hat{\mathbb{P}}_n\left(#1\right)}
\newcommand{\var}[1]{\mbox{Var}\left(#1\right)}
\newcommand{\expec}[1]{\mathbb{E}\left(#1\right)}
\newcommand{\probER}[1]{\mathbb{Q}_n\left(#1\right)}
\newcommand{\abs}[1]{\left|#1\right|}
\newcommand{\sigtil}{\widetilde{\sigma}}
\newcommand{\tautil}{\widetilde{\tau}}
\newcommand{\Ntil}{\widetilde{N}}
\newcommand{\Xtil}{\widetilde{X}}
\newcommand{\xitil}{\widetilde{\xi}}

\newcommand{\bigoh}[1]{O\left(#1\right)}
\newcommand{\trans}[1]{{#1}^{\intercal}}
\newcommand{\set}[1]{\left\{#1\right\}}
\newcommand{\defas}{:=}
\renewcommand{\1}{\mathbf{1}}
\newcommand{\indicator}[1]{\mathbbm{1}_{\{#1\}}}

\DeclareMathOperator{\Poisson}{Poisson}
\DeclareMathOperator{\overlap}{overlap}

\newtheorem*{rep@theorem}{\rep@title}
\newcommand{\newreptheorem}[2]{%
\newenvironment{rep#1}[1]{%
 \def\rep@title{#2 \ref{##1}}%
 \begin{rep@theorem}}%
 {\end{rep@theorem}}}
\newreptheorem{theorem}{Theorem}

\newtheorem*{rep@proposition}{\rep@title}
\newcommand{\newrepproposition}[2]{%
\newenvironment{rep#1}[1]{%
 \def\rep@title{#2 \ref{##1}}%
 \begin{rep@proposition}}%
 {\end{rep@proposition}}}
\newreptheorem{proposition}{Proposition}

\newtheorem*{rep@lemma}{\rep@title}
\newcommand{\newreplemma}[2]{%
\newenvironment{rep#1}[1]{%
 \def\rep@title{#2 \ref{##1}}%
 \begin{rep@lemma}}%
 {\end{rep@lemma}}}
\newreptheorem{lemma}{Lemma}

\title{Non-Reconstructability in the Stochastic Block Model}
\author[1,2]{Joe Neeman
\thanks{joeneeman@gmail.com}}
\author[1]{Praneeth Netrapalli
\thanks{praneethn@utexas.edu}}
\affil[1]{Dept. of Electrical and Computer Engg., UT Austin}
\affil[2]{Dept. of Mathematics, UT Austin}
\begin{document}
\maketitle
\begin{abstract}
  We consider the problem of clustering (or reconstruction) in the stochastic block model, in the regime where the
average degree is constant.
For the case of two clusters with equal sizes, recent results \cite{MoNeSl:13,Massoulie:13,MoNeSl:14} show that reconstructability undergoes a phase
transition at the Kesten-Stigum bound of $\lambda_2^2 d = 1$, where $\lambda_2$ is the second largest eigenvalue
of a related stochastic matrix and $d$ is the average degree. In this paper, we address the general case of
more than two clusters and/or unbalanced cluster sizes. Our main result is a sufficient condition for clustering to be
impossible, which matches the existing result for two clusters of equal sizes. A key ingredient in our result is a new
connection between non-reconstructability and non-distinguishability of the block model from an Erd\H{o}s-R\'enyi model
with the same average degree.
We also show that it is some times possible to reconstruct even when $\lambda_2^2 d < 1$.
Our results provide evidence supporting a series of conjectures made by Decelle et al. \cite{Z:2011}
regarding reconstructability and distinguishability of stochastic block models (but do not settle them).
\end{abstract}

 \section{Introduction}
 Stochastic block models are popular models for random graphs that exhibit community
 structures. In these models, the vertices of a graph are divided into at least two
 different classes, and then edges are added between vertices with probabilities
 that depend on the classes of the vertices. We will consider sparse stochastic
 block models, where the average degree of each vertex is constant, while the number
 of vertices tends to infinity.
 
 The fundamental problem in stochastic block models is the \emph{community detection},
 or \emph{reconstruction} problem: if we are given a graph from a stochastic block model,
 but we are not told which vertex belongs to which class, then can we recover this
 information from the structure of the graph?
 In the sparse regime, straightforward
 probabilistic arguments show that it is not possible to correctly classify more
 than a certain fraction of the vertices correctly (unlike in the denser case, where it
 is sometimes possible to completely reconstruct the classes). Given this restriction,
 we say that a stochastic block model is reconstructible if there is some algorithm
 that recovers the classes more accurately than a random guess would.
 
 Decelle et al.~\cite{Z:2011} stimulated the study of block model reconstructability with
 a series of striking conjectures, which were backed up by simulations and connections
 to statistics physics. They proposed three main scenarios: i) if the ``signal'' in the model
 is strong enough compared to the ``noise,'' then the model can be reconstructed efficiently,
 ii) if the signal is somewhat weaker, then the model is reconstructible, but it is computationally
 hard to actually carry out the reconstruction, and finally, iii) if the signal is too weak then the
 model is not reconstructible. Decelle et al.\ further divide this third scenario into two parts,
 depending on whether or not belief propagation has a fixed point that is correlated with
 the true classification.
 
 Besides merely conjecturing the existence of these various scenarios, Decelle et al.\ gave
 estimates for the boundaries between them. Specifically, they conjectured that
 the boundary between the efficiently reconstructible and the computationally hard regions
 should always occur at the \emph{Kesten-Stigum bound}, while the reconstructibility
 boundary corresponds to the condensation transition in spin glasses. In the special case
 of two classes with equal sizes, these conjectures have been verified by Mossel et al.~\cite{MoNeSl:13,MoNeSl:14}
 and Massouli\'e~\cite{Massoulie:13}; in this special case, the two transitions coincide and so there is
 no reconstructible-but-hard regime.
 
 Besides the reconstruction problem, it is natural to ask whether a stochastic block model
 can be distinguished from an Erd\H{o}s-R\'enyi model with the same average degree.
 This is called the \emph{distinguishability} problem, and it is intuitively easier than
 the reconstruction problem in the sense that an algorithm that finds communities in
 a block model should also be able to check whether there exist meaningful communities.
 However, we are not aware of a rigorous version of this statement; indeed, one of our
 main results shows that non-reconstructibility follows from something somewhat stronger
 than non-distinguishability (specifically, the existence of a certain second moment).
 
 In addition to showing a connection between reconstruction and distinguishability,
 we explore the distinguishability problem in more detail. Specifically, we show that
 above the Kesten-Stigum bound a block model is orthogonal to the corresponding
 Erd\H{o}s-R\'enyi model; on the other hand, the two models are mutually contiguous
 sufficiently far below the Kesten-Stigum bound. Except in the special case of two equal-sized
 classes (which was previously considered by Mossel et al.~\cite{MoNeSl:13}), we do not know whether the
 contiguity region extends all the way to the Kesten-Stigum bound -- indeed we exhibit
 models with two highly unbalanced classes, where the contiguity region does not
 extend all the way to the Kesten-Stigum bound. We do derive
 sharp bounds for the region on which a certain second moment is finite, but the second
 moment condition in question is only sufficient, but not necessary, to prove contiguity.

 \section{Background and related work}
 Stochastic block models were introduced by Holland, Laskey and Leinhardt~\cite{HLL:83} in the 1980s
 in order to study community detection in random networks, and then re-discovered
 independently by Dyer and Frieze~\cite{DF:89} in the context of computational complexity.
 Their work and subsequent work by -- to mention a few -- Jerrum and Sorkin~\cite{JS:98},
 Condon and Karp~\cite{CK:01}, and Snijders and Nowicki~\cite{SN:97} focused primarily on the regime
 in which the graph is sufficiently dense, and the signal sufficiently strong, to
 exactly recover the communities with high probability.

 More recent work has considered very sparse block models, where the average degree of
 each vertex is bounded as the number of vertices increases.
 In this regime,
 the first algorithm with provable guarantees was given by Coja-Oghlan~\cite{CO:10}. More
 recently, Massouli\'e~\cite{Massoulie:13} and Mossel et al.~\cite{MoNeSl:14} gave algorithms with better
 guarantees, but only in the case of two classes with approximately equal size.
 The case of two classes with very unbalanced sizes was addressed recently by
 Verzelen and Arias-Castro~\cite{VerzelenAriasCastro:14}; they do not study the problem of reconstructing
 communities, but only the problem of detecting whether a community exists.
 The sparse regime has also seen work without proven performance guarantees: Decelle et al.~\cite{Z:2011}
 gave an algorithm based on belief propagation, while Krzakala et al.~\cite{Krzakala_etal:13} proposed
 a spectral algorithm.
 \section{Definitions and results}
 
 A stochastic block model with $s \ge 2$ communities is parametrized by two quantities: the distribution
 $\pi \in \Delta_{s-1}$ of vertex classes and the symmetric matrix $M \in \R^{s \times s}$ of edge
 probabilities. Given these two parameters, a random graph from the block model
 $\calG(n, M, \pi)$ is sampled as follows: for each vertex $v$, sample
 a label $\sigma_v$ in $[s] = \{0, 1, \dots, s-1\}$ independently with distribution $\pi$. Then,
 for each pair $(u, v)$, include the edge $(u, v)$ in the graph independently with probability
 $n^{-1} M_{\sigma_u,\sigma_v}$. Since we will work with a fixed $M$ and $\pi$ throughout,
 we denote $\calG(n, M, \pi)$ by $\P_n$. Note that according to the preceding description,
 we have the following explicit form for the density of $\P_n$:
 \[
  \P_n(G, \sigma) = \prod_{v \in V(G)} \pi_{\sigma_v}
  \prod_{(u, v) \in E(G)} \frac{M_{\sigma_u,\sigma_v}}{n}
  \prod_{(u, v) \not \in E(G)} \left(1 - \frac{M_{\sigma_u,\sigma_v}}{n}\right).
 \]
 We will assume throughout that every vertex in $G \sim \P_n$ has the same expected degree.
 (In terms of $M$ and $\pi$, this means that $\sum_j M_{ij} \pi_j$ does not depend on $i$.)
 Without this assumption, reconstruction and distinguishability -- at least in the way
 that we will phrase them -- are trivial, since we gain non-trivial information on
 the class of a vertex just by considering its degree.
 
 With the preceding assumption in mind, let $d = \sum_j M_{ij} \pi_j$ be the expected
 degree of an arbitrary vertex. In order to discuss distinguishability, we will compare
 $\P_n$ with the Erd\H{o}s-R\'enyi distribution $\Q_n := \calG(n, d/n)$.

 Throughout this work, we will make use of the matrix $T$ defined by
 \[
  T_{ij} = \frac 1d \pi_i M_{ij},
 \]
 or in other words, $T = \frac 1d \diag(\pi) M$. Note that $T$ is a stochastic matrix, in the
 sense that it has non-negative elements and all its rows sum to 1. The Perron-Frobenius eigenvectors
 of $T$ are $\pi$ on the left, and $\1$ on the right (where $\1$ denotes the vector of ones),
 and the corresponding eigenvalue is 1. We let $\lambda_1, \dots, \lambda_s$ be the eigenvalues of
 $T$, arranged in order of decreasing absolute value (so that $\lambda_1 = 1$ and $|\lambda_2| \le 1$).
 
 There is an important probabilistic interpretation of the matrix $T$ relating to the local
 structure of $G \sim \P_n$; although we will not rely on this interpretation
 in the current work, it played an important role in~\cite{MoNeSl:13}. Indeed, using
 an argument similar to the one in~\cite{MoNeSl:13}, one can show that for any fixed radius $R$,
 the $R$-neighborhood of a vertex in $G \sim \P_n$ has almost the same distribution as a Galton-Watson
 tree with radius $R$ and offspring distribution $\Poisson(d)$. Then, the class labels on the
 neighborhood can be generated by first choosing the label of the root according to $\pi$ and then,
 conditioned on the root's label being $i$, choosing its children's labels independently to be
 $j$ with probability $T_{ij}$. This procedure continues down the tree: any vertex with parent $u$
 has probability $T_{\sigma_u j}$ to receive the label $j$. Thus, $T$ is the transition matrix
 of a certain Markov process that describes a procedure for approximately generating a local
 neighborhood in $G$.
 
 \subsection{Positive results}
 \subsubsection{Distinguishability}
 It is not hard to show that if $d \lambda_2^2 > 1$ then the block model $\P_n$ is
 asymptotically orthogonal to the Erd\H{o}s-R\'enyi model $\Q_n$, in the sense that there is
 a sequence of events $\Omega_n$ such that $\P_n(\Omega_n) \to 1$ and $\Q_n(\Omega_n) \to 0$.
 Indeed, this statement follows fairly easily from the following cycle-counting result,
 due to Bollob\'as et al.~\cite{BoJaRi:07}
 
\begin{proposition}\label{prop:cycle-count}
Let $X_k$ be the number of $k$-cycles in $G$. Then
\begin{align*}
 X_k &\toD \Pois\left(\frac 1{2k} d^k\right) \text{ under $\Q_n$, and}\\
 X_k &\toD \Pois\left(\frac 1{2k} d^k \tr(T^k)\right) \text{ under $\P_n$.}
\end{align*}
\end{proposition}

 We remark that the $X_k$ are also asymptotically independent, in the sense that for any
 positive numbers $j_3, \dots, j_m$, the moment
 $\E \prod_{i=3}^m X_i^{j_i}$ converges as $n \to \infty$ to $\E \prod_{i=3}^m Y_i^{j_i}$,
 for independent Poisson variables $Y_3, \dots, Y_m$.

 Now, Chebyshev's inequality implies that $X_k \le \frac{1}{2k} d^k + O(k^{-1/2} d^{k/2})$ under
 $\Q_n$, while $X_k \ge \frac{1}{2k} d^k \tr(T^k) - O(k^{-1/2} d^{k/2})$ under $\P_n$. If
 $d \lambda_2^2 > 1$, then these ranges are disjoint for large enough $k$; in particular,
 the event $\Omega_n = \{X_k \ge \frac{1}{2k} d^k \tr(T^k) - O(k^{-1/2} d^{k/2})\}$ satistifes
 $\P_n(\Omega_n) \to 1$ and $\Q_n(\Omega_n) \to 0$. The details of this argument are contained
 in~\cite{MoNeSl:13} in the case $s = 2$, $\pi = (1/2, 1/2)$, but exactly the same argument applies
 in the general case:
 
 \begin{theorem}\label{thm:distinguish}
  If $d \lambda_2^2 > 1$ then $\P_n$ and $\Q_n$ are asymptotically orthogonal.
 \end{theorem}

  \subsubsection{Reconstructability}
  Let $\sigma$ and $\tau$ denote labellings in $[s]^n$, and let
  \begin{align*}
  N_i(\sigma) &= \#\{v : \sigma_v = i\} \\
  N_{ij}(\sigma, \tau) &= \#\{v : \sigma_v = i, \tau_v = j\}.
  \end{align*}
  We define the \emph{overlap} between $\sigma$ and $\tau$ by
  \[
   \overlap(\sigma, \tau) = \frac{1}{n}\max_\rho \sum_{i=1}^s \left(N_{i \rho(i)} (\sigma, \tau) - \frac 1n N_i(\sigma) N_{\rho(i)}(\tau)\right),
  \]
  where the supremum runs over all permutations $\rho$ of $[s]$. In words, $\sigma$ and $\tau$
  have a positive overlap if there is some relabelling of $[s]$ so that they are positively correlated.

  We say that the block model $\P_n = \calG(n, M/n, \pi)$ is \emph{reconstructable}
  if there is some $\delta > 0$ and an algorithm $\calA$ mapping graphs to labellings
  such that if $(G, \sigma) \sim \P_n$ then
  \[
   \lim_{n \to \infty} \Pr(\overlap(\calA(G), \sigma) > \delta) > 0.
  \]

  In other words, we are looking for an algorithm that guarantees a non-trivial overlap, with
  a non-trivial probability, as $n \to \infty$.
  
  Mossel et al.~\cite{MoNeSl:13,MoNeSl:14} and Massouli\'e~\cite{Massoulie:13} show that for balanced two cluster case,
  $d \lambda_2^2 = 1$ is the threshold for reconstructability. It was not known if
  this is also the threshold for unbalanced two cluster case. The following result shows that
  this is not the case.
  \begin{proposition}\label{prop:reconstruction-belowKS}
    For every $\epsilon > 0$, there exist (unbalanced) $2$-cluster models with $d\lambda_2^2 < \epsilon$ where reconstruction is possible.
  \end{proposition}
  The proof is a simple application of Bernstein's inequality and can be found in Section~\ref{sec:rec-belowKS}.
  We remark, however, that our reconstruction algorithm for proving Proposition~\ref{prop:reconstruction-belowKS}
  is not computationally efficient. Indeed, if Decelle et al.'s~\cite{Z:2011} are correct then no
  computationally efficient reconstruction is possible when $d\lambda_2^2 <1$.

 \subsection{Negative results}
 
 At the other extreme of asymptotic orthogonality is contiguity: we say that
 $\P_n$ is asymptotically contiguous to $\Q_n$ if for any sequence of events $\Omega_n$,
 $\Q_n(\Omega_n) \to 1$ implies $\P_n(\Omega_n) \to 1$. We say that $\P_n$
 and $\Q_n$ are mutually asymptotically contiguous if $\P_n$ is asymptotically contiguous to $\Q_n$ and
 $\Q_n$ is asymptotically contiguous to $\P_n$. From the statistical perspective of distinguishing
 $\P_n$ and $\Q_n$, mutual asymptotic contiguity implies that no test could ever be sure whether
 a given sample came from $\P_n$ or $\Q_n$.
 
 In the case $s = 2$, $\pi = (1/2, 1/2)$, Mossel et al.~\cite{MoNeSl:13} gave a converse to
 Theorem~\ref{thm:distinguish}: they showed that if $d \lambda_2^2 < 1$ then $\P_n$ and
 $\Q_n$ are mutually asymptotically contiguous. In the general case, we still lack a sharp converse
 to Theorem~\ref{thm:distinguish} (indeed the proof of Proposition~\ref{prop:reconstruction-belowKS}
 shows that there can not be one). Nevertheless, we give a sufficient condition for
 $\P_n$ and $\Q_n$ to be mutually asymptotically contiguous.
 
 \subsubsection{The uniform integrability condition}
 
 At a crucial point in our analysis, we require that the exponential of a certain
 multinomial quadratic form be uniformly integrable. This is the main place in which our analysis
 differs from the special case considered in~\cite{MoNeSl:13}: in the case of two balanced
 classes, Mossel et al.\ were led to consider $\exp(\lambda Z_n^2)$,
 where $Z_n = 2 n^{-1/2} (\Binom(n, 1/2) - n/2)$. This is a particularly nice special case
 because $\exp(\lambda Z_n^2)$ is uniformly integrable for all $\lambda < 1/2$, which is
 exactly the set of $\lambda$ for which $\E \exp(\lambda Z^2) < \infty$, where $Z \sim \normal(0, 1)$.
 The situation is more complicated for general binomial and multinomial variables, and it
 leads us to the following definitions:
  \begin{definition}\label{def:D}
 Let $\Delta_k$ to denote the simplex in $k$-dimensions,
\begin{align*}
  \Delta_k \defas \set{p \in \R^k:  p_i \geq 0, \sum_{i=1}^k p_i = 1}.
\end{align*}
   Define $D: \Delta_k \times \Delta_k \to \R$ by
  \[
   D(p, q) = \sum_{i=1}^k p_i \log (p_i/q_i).
  \]
 \end{definition}
 
Note that if we interpret $p, q \in \Delta_k$ as probability distributions on a $k$-point
 set, then $D(p, q)$ is exactly the Kullback-Leibler divergence between them.
 
 \begin{definition}
  For $\pi \in \Delta_s$, define
  \[
   \Delta_{s^2}(\pi) \defas \set{
   (p_{ij})_{i,j=1}^s \in \Delta_{s^2}: \sum_{i=1}^s p_{ij} = \pi_j \text{ and } \sum_{j=1}^s p_{ij} = \pi_i
   \text{ for all } i,j
   }.
  \]
  In other words, elements of $\Delta_{s^2}(\pi)$ are probability distributions on $[s]^2$ that have $\pi$ as
  their marginal distributions.
 \end{definition}

 \begin{definition}\label{def:Q}
  For $\pi \in \Delta_s$ and an $s \times s$ matrix $A$, let $p = \pi \otimes \pi$  and define
  \[
   Q(\pi, A) = \sup_{\alpha \in \Delta_{s^2}(\pi)} \frac{(\alpha - p)^T (A \otimes A) (\alpha - p)}{D(\alpha, p)}.
  \]
 \end{definition}
 
 The preceding definition may not seem well-motivated, but we will show that $Q(\pi, A) < 1$ is exactly the
 right condition for a certain exponentiated quadratic form involving $A$ to be uniformly integrable. Moreover, although
 we do not know any simple algebraic expression for $Q$, one can easily compute numerical approximations;
 see Figure~\ref{fig:unbalanced2clusters}.
 
 \subsubsection{Non-distinguishability}
 \begin{theorem}\label{thm:non-distinguish}
  Let $\P_n = \calG(n, M/n, \pi)$ and $\Q_n = \calG(n, d/n)$, where $d = \sum_{j} M_{ij} \pi_j$.
  Define $A = M - d \1\1^T$.
  If $Q(\pi, A/\sqrt{2d}) < 1$ then $\P_n$ and $\Q_n$ are mutually contiguous.
 \end{theorem}
 
 For comparison with Theorem~\ref{thm:distinguish}, note that $Q(\pi, A/\sqrt{2d}) < 1$
 implies that $\lambda_2^2 d < 1$. This comes from comparing the second derivatives in
 the numerator and denominator of $Q$: if $Q < 1$ then the Hessian of the numerator must
 be smaller (in the semidefinite order) than the denominator, and this turns out to be equivalent
 to $\lambda_2^2 d < 1$.
 
 We remark that while $Q(\pi, A/\sqrt{2d}) < 1$ is only a sufficient condition for the contiguity
 of $\P_n$ and $\Q_n$, it is actually a sharp condition for a certain second moment to exist:
 \begin{theorem}\label{thm:second-moment}
 Fix a sequence $a_n$ with $a_n = o(n)$ and $a_n = \omega(\sqrt n)$. Let $\Omega_n$
 be the event that for all $i \in [s]$, $|\set{u : \sigma_u = i}| = n\pi_i \pm a_n$.
 With the notation of Theorem~\ref{thm:non-distinguish},
 take $\tilde \P_n$ to be $\P_n$ conditioned on $\Omega$.
 If $Q(\pi, A/\sqrt{2d}) < 1$ then
 \begin{equation}\label{eq:second-moment}
\lim_{n\to\infty} \E_{\Q_n} \left(\diff{\tilde \P_n}{\Q_n}\right)^2 = (1 + o(1)) \prod_{i,j=2}^s \psi(d\lambda_i \lambda_j)
< \infty,
 \end{equation}
 where $\psi(x) = (1-x)^{-1/2} e^{-x/2-x^2/4}$.
 On the other hand, if $Q(\pi, A/\sqrt{2d}) > 1$ then
   \[\lim_{n\to\infty} \E_{\Q_n} \left(\diff{\tilde \P_n}{\Q_n}\right)^2 = \infty.\]
 \end{theorem}
  
  In fact, Theorem~\ref{thm:second-moment} is the important technical step in the proof
  of Theorem~\ref{thm:non-distinguish}; it is easy to see that~\eqref{eq:second-moment}
  implies that $\P_n$ is asymptotically contiguous to $\Q_n$; the other direction
  (i.e., $\Q_n$ is asymptotically contiguous to $\P_n$) follows from
  a conditional second-moment argument, of which~\eqref{eq:second-moment} is the
  most challenging step.
  
  We remark that one can also prove a version of Theorem~\ref{thm:second-moment} without conditioning
  on $\Omega_n$; however, one would need to replace $\Delta_{s^2}(\pi)$ in Definition~\ref{def:Q} by
  the larger set $\Delta_{s^2}$. This turns out to increase $Q$, and therefore gives a weaker result.
  In other words, there is a regime in which
  \[
    \E_{\Q_n} \left(\diff{\P_n}{\Q_n}\right)^2
  \]
  tends to infinity, but only because the integrand explodes on the rare event that the labelling $\sigma$
  is very unbalanced.
  
  \subsubsection{Non-reconstructability}  
  The following theorem is our main result on non-reconstructability.
  \begin{theorem}\label{thm:non-reconstruct}
   With the notation of Theorem~\ref{thm:non-distinguish}, if $Q(\pi, A/\sqrt{2d}) < 1$ then
   $\P_n$ is not reconstructable.
  \end{theorem}
  Existing results on non-reconstructability of the balanced two cluster model have been obtained by reducing the
  problem to one of non-reconstructability on trees \cite{MoNeSl:13}. However, finding the non-reconstructable
  region of trees
  in the more general case has been a long standing open problem.
  Instead, we obtain Theorem~\ref{thm:non-reconstruct} by showing a connection between
  distinguishability and reconstructability.
  Intuitively, detecting communities in $G$ seems harder than merely distinguishing $G \sim \P_n$
  from $G \sim \Q_n$; however, it is not known whether Theorem~\ref{thm:non-distinguish}
  implies Theorem~\ref{thm:non-reconstruct}. Instead, we give a reduction from Theorem~\ref{thm:second-moment}:
  we show that the condition~\eqref{eq:second-moment} implies non-reconstructability.

  \subsection{Numerical results}

  We present some numerical description of Theorem~\ref{thm:non-distinguish}'s uniform integrability condition in
  the case of two clusters with unequal sizes.
  In the case of two clusters, for a fixed probability vector $\pi$,
  the uniform integrability condition
  turns out to be just a threshold on $\lambda_2^2 d$. To see this, we first
  note that the matrix $A$ is rank-$1$ and hence, so is $A^{\otimes 2}$.
  Moreover, fixing $\pi$ also fixes the eigenvector of $A$, and so it fixes the numerator
  of $Q$ up to a scaling.
  On the other hand, for a fixed $\pi$, the denominator of $Q$ is a function only of $\alpha$.
  So, we see that:
  \begin{align*}
  Q(\pi, A/\sqrt{2d}) = c(\pi) \lambda_2^2 d \sup_{\alpha \in \Delta_{s^2}(\pi)} \frac{\abs{\trans{(\pi^{\otimes 2} - \alpha)} a^{\otimes 2} }^2}{2D(\alpha, \pi^{\otimes 2})},
  \end{align*}
  where $a$ is the unit-eigenvector of $A$.
  Figure~\ref{fig:unbalanced2clusters} shows how the threshold on $\lambda_2^2 d$ varies with $\pi$.
  \begin{figure}[th]
  \centering  \includegraphics[width=0.7\textwidth]{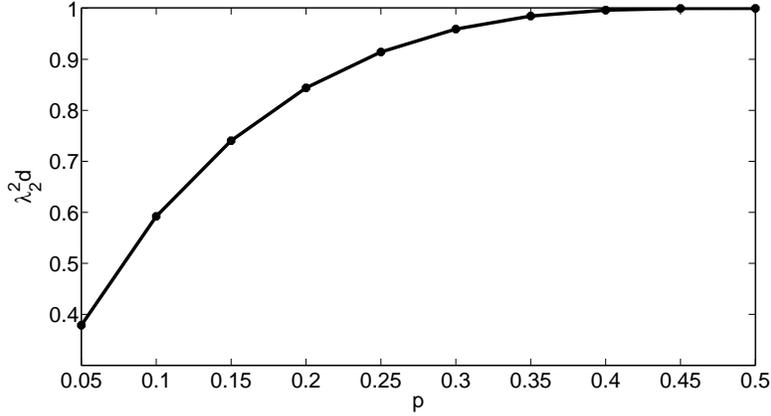}
  \caption{This plot shows how the threshold on $\lambda_2^2 d$ from the uniform integrability condition varies with the probability
  vector $\pi$ in the two cluster case. The x-axis shows $p$ where 
  $\pi = [p,1-p]$ and the y-axis shows the value of the threshold
  computed using numerical optimization. The plot shows that our bound is tight as the clusters become balanced (where the
  threshold is close to $1$). However, the threshold decreases as the clusters get more unbalanced.}
  \label{fig:unbalanced2clusters}
  \end{figure}

\subsection{An example showing looseness}\label{sec:intro-examples}
Figure~\ref{fig:unbalanced2clusters} shows that as the clusters become very unbalanced (i.e., $p \rightarrow 0$),
Theorem~\ref{thm:non-reconstruct} only guarantees non-reconstructability when $d \lambda_2^2$ is very small.
One might ask whether the true reconstructability threshold has this behavior.
Proposition~\ref{prop:reconstruction-belowKS} shows that it does for some models.
We now present another family of examples for which the behavior of the threshold is quite different:
\begin{proposition}\label{prop:ex1}
Consider the block model given by
\begin{align*}
 M = d \begin{pmatrix}
	  \frac{1}{p} & 0 \\
	  0 & \frac{1}{1-p}
       \end{pmatrix}.
\end{align*}
If $d < 1$ then the above model is not reconstructable.
\end{proposition}
Note that the above block model has $\lambda_2 = 1$. Hence Proposition~\ref{prop:ex1} shows that there exist arbitrarily
unbalanced models, where we can not reconstruct for $d \lambda_2^2 < 1$.

The main idea behind Proposition~\ref{prop:ex1} is that for $d<1$, the largest component is of size $\bigoh{\log n}$.
Even if we can reconstruct the labels of nodes in the same cluster very well, we can not predict the labels of nodes in
different clusters better than random guessing.
We use a lemma from Mossel et al. \cite{MoNeSl:13} that captures this intuition.

\section{Reconstruction below the Kesten-Stigum bound}\label{sec:rec-belowKS}
In this section, we prove Proposition~\ref{prop:reconstruction-belowKS}.
We consider a $2$-cluster model with

\begin{center}
\begin{tabular}{cc}
  $M \defas d \begin{pmatrix}
  a & b \\
  b & c
\end{pmatrix},$ & $\pi \defas \trans{[p \;1-p]}$,
\end{tabular}
\end{center}
with $pa+(1-p)b = pb+(1-p)c = 1$. The second eigenvalue of $T$ turns out to be
\begin{align*}
  \lambda_2 = \frac{(a-1) p}{1-p}.
\end{align*}

Our first lemma shows that if two assignments have very small overlap, then they are almost uncorrelated.
\begin{lemma}\label{lem:overlap-correlation}
Suppose $\sigma$ and $\tau$ are $(pn,(1-p)n)$ and $(qn,(1-q)n)$ partitions respectively.
Define
\begin{align*}
  p_1 \defas \frac{\abs{\set{u \middle\vert \sigma_u = \tau_u = 1}}}{qn}.
\end{align*}
If $\overlap(\sigma,\tau) < 2q \epsilon$ then $\abs{p_1 - p} < \epsilon$.
\end{lemma}
\begin{proof}
We prove by contradiction. Suppose $p_1 \geq p + \epsilon$.
Let $p_2 \defas \frac{\abs{\set{u \middle\vert \sigma_u = 1, \tau_u = 2}}}{(1-q)n}$.
Since $p_1 q + p_2 (1-q) = p$, we get $p_2 = \frac{p - p_1 q}{1-q}$.
We have
\begin{align*}
  \overlap(\sigma, \tau) &= \frac{1}{n}\max_\rho \sum_{i=1}^2 \left(N_{i \rho(i)} (\sigma, \tau) - \frac 1n N_i(\sigma) N_{\rho(i)}(\tau)\right) \\
	&\geq \frac{1}{n} \sum_{i=1}^2 \left(N_{i i} (\sigma, \tau) - \frac 1n N_i(\sigma) N_{i}(\tau)\right) \\
	&= (p_1 q - pq) + ((1-p_2)(1-q)-(1-p)(1-q)) \\
	&= 2q (p_1 - p) \geq 2q \epsilon.
\end{align*}
This is a contradiction. Similarly, we can show that $p_1 \geq p - \epsilon$.
\end{proof}
We are now ready to prove Proposition~\ref{prop:reconstruction-belowKS}. Its proof is a simple application of
Bernstein's inequality.
\begin{proof}[Proof of Proposition~\ref{prop:reconstruction-belowKS}]
The reconstruction algorithm is an exhaustive search over all $(pn,(1-p)n)$ partitions.
For each such partition, it looks at the number of edges within the $pn$ block.
If the number of edges is in $\left[\left(\frac{dap^2}{2}-\delta\right)n,\left(\frac{dap^2}{2}+\delta\right)n\right]$,
then it outputs (one such) partition. It is easy to show that the true partition will satisfy the above property
with high probability and so,
there is at least one partition that the algorithm can output. The rest of the proof is to show that
any partition which has low enough overlap with the true partition does not satisfy the above property.

Let $\sigma$ be the true partition. Let $\tau$ be a $(pn,(1-p)n)$ partition such that $\overlap(\sigma,\tau) < \delta$.
Let $S_1 \defas \set{u \middle\vert \tau_u = 1}$, and
$p_1 \defas \frac{\abs{\set{u \middle\vert \sigma_u = \tau_u = 1}}}{pn}$.
From Lemma~\ref{lem:overlap-correlation}, we see that $\abs{p_1 - p} < \bigoh{\delta}$.
Consider the random variable
\begin{align*}
E_1 \defas \sum_{\stackrel{u,v \in S_1}{u < v}} X_{uv} = \# \mbox{ edges within } S_1 - \left(\frac{dp^2}{2}+\bigoh{\delta}\right) n,
\end{align*}
where $X_{uv} = \indicator{uv \in G} - \frac{d M_{\sigma_u \sigma_v}}{n}$.
We see that
\begin{itemize}
  \item	$\expec{X_{uv}}=0$,
  \item	$\var{E_1} = \left(\frac{dp^2}{2}+\bigoh{\delta}\right) n$, and
  \item	$X_{uv} < 1$ a.s.
\end{itemize}
Using Bernstein's inequality, we have:
\begin{align*}
\P_n\left[\abs{E_1} > \left(\frac{d \abs{a-1}p^2}{2}+\bigoh{\delta}\right)n \middle\vert \sigma\right]
	&< \exp\left(\frac{-\frac{1}{2}\left(\frac{d \abs{a-1}p^2}{2}+\bigoh{\delta}\right)^2 n^2}
	  {\left(\frac{dp^2}{2}+\bigoh{\delta}\right) n + \frac{1}{3} \left(\frac{d \abs{a-1}p^2}{2}+\bigoh{\delta}\right)n }\right) \\
	&\leq \exp\left(\frac{-3}{4} \left(\frac{d(a-1)^2p^2 }{a+3} + \bigoh{\delta}\right)n\right) \\
	&= \exp\left(\frac{-3}{4} \left(\frac{d\lambda_2^2 (1-p)^2 }{a+3} + \bigoh{\delta}\right)n\right)
\end{align*}
So with high probability, $\tau$ will not be output by the algorithm. Since each such $\tau$ is a $(pn,(1-p)n)$ partition, the total number
of such $\tau$ is at most $\exp\left((H(p)+o(1))n\right)$, where $H(\cdot)$ is the entropy function,
$H(p) \defas p \log \frac{1}{p} + (1-p) \log \frac{1}{1-p}$. If
\begin{align*}
\frac{3d\lambda_2^2 (1-p)^2 }{4(a+3)} > H(p),
\end{align*}
then taking a small enough $\delta$ and union bound gives us the result.
In other words, all we need is
\begin{align*}
\lambda_2^2 d > \frac{4(a+3) H(p)}{3(1-p)^2}.
\end{align*}
The right hand side can be made as small as we wish by choosing, say $a = 2$ and $p$ small enough.
\end{proof}

\section{Non-distinguishability}

In this section, we will assume Theorem~\ref{thm:second-moment} and use
it to prove Theorem~\ref{thm:non-distinguish}. Our main tool is the conditional second moment
method, which was originally developed by Robinson and Wormald~\cite{RW:92} in their study of Hamiltonian
cycles in $d$-regular graphs. Janson~\cite{Janson:95} was the first to apply this method for proving contiguity.
We use a formulation from~\cite[Theorem 4.1]{Wormald}:

\begin{theorem}\label{thm:conditional-second-mom}
 Consider two sequences $\P_n, \Q_n$ of probability distributions.
 Suppose that there exist random variables $\{X_{k,n}: k \ge 3\}$ such that for every $k$,
 \begin{align}
  X_{k,n} &\toD \Pois(\mu_k) \text{ under $\Q_n$, as $n \to \infty$; and} \label{eq:cycles-Qn} \\
  X_{k,n} &\toD \Pois(\mu_k(1 + \delta_k)) \text{ under $\P_n$, as $n \to \infty$.}\label{eq:cycles-Pn}
 \end{align}
 Suppose also that for any $k^*$, the collection $X_{3,n}, \dots, X_{k^*,n}$ are asymptotically independent
 as $n \to \infty$, under both $\P_n$ and $\Q_n$.
 If
 \begin{equation}\label{eq:second-moment-cond}
  \E_{\Q_n} \left(\frac{d \P_n}{d \Q_n}\right)^2 \le (1 + o(1)) \exp\left(\sum_{k \ge 3} \mu_k \delta_k^2\right)
  < \infty
 \end{equation}
 then $\P_n$ and $\Q_n$ are mutually contiguous.
\end{theorem}

We will apply Theorem~\ref{thm:conditional-second-mom} with $\P_n$ replaced by
$\hat \P_n = (\P_n \mid \Omega_n)$; i.e., the
block model conditioned on having almost the expected label frequencies.

We note that~\eqref{eq:cycles-Qn},~\eqref{eq:cycles-Pn}, and the asymptotic independence property are already verified by
Proposition~\ref{prop:cycle-count}, with $\mu_k = \frac{1}{2k} d^k$ and
$\delta_k = \tr(T^k) - 1$. Recalling that $1 = \lambda_1 \ge \cdots \ge \lambda_s$ are the eigenvalues of $T$, we have
$\delta_k = \sum_{i \ge 2} \lambda_i$. Hence,
\begin{align*}
 \sum_{k=3}^\infty \mu_k \delta_k^2
 &= \frac 12 \sum_{k=3}^\infty \frac {d^k}{k} \sum_{i,j=2}^s \lambda_i^k \lambda_j^k \\
 &= \frac 12 \sum_{i,j=2}^s \sum_{k=3}^\infty \frac {(d\lambda_i \lambda_j)^k}{k} \\
 &= \sum_{i,j=2}^s \log \psi(d\lambda_i \lambda_j),
\end{align*}
where $\psi(x) = (1-x)^{-1/2} e^{-x/2-x^2/4}$. In particular, condition~\eqref{eq:second-moment-cond}
follows immediately from Theorem~\ref{thm:second-moment}, which in turn proves
that $\tilde \P_n$ and $\Q_n$ are mutually contiguous. Since
$\P_n(\Omega_n) \to 1$, $\P_n$ and $\tilde \P_n$ are mutually
contiguous, and Theorem~\ref{thm:non-distinguish} follows.

\section{Second moment}
In this section, we will prove our main result calculating the second moment under the uniform integrability condition (Theorem~\ref{thm:second-moment}).
Our first lemma expresses the second moment in terms of (centered and normalized) multinomial random variables.
In order to state the lemma, we make the following notation.
Given two assignments $\sigma, \tau \in [s]^n$, let $N_{ij} \defas N_{ij}(\sigma, \tau) \defas |\{v : \sigma_v = i, \tau_v = j\}|$, and
$X_{ij} \defas X_{ij}(\sigma, \tau) \defas n^{-1/2}\left(N_{ij}- n \pi_i\pi_j\right)$.
Recall that $\Omega_n$ is the event that the label frequencies are approximately their expected values, and
let $Y_n$ denote the restricted density $\indicator{\Omega_n} \frac{d \P_n}{d \Q_n}$.
 With a slight overlading of notation, we write $\sigma \in \Omega_n$ if for all $i \in [s]$, 
 $|\{u : \sigma_u = i\}| = n \pi_i \pm a_n$.
Recall that $A \defas M - d \1\trans{\1}$.

\begin{lemma}\label{lem:second-moment-simpl}
We have:
\begin{align*}
  \E_{\Q_n} Y_n^2
  &= (1 + O(n^{-1}))
 \sum_{\sigma,\tau \in \Omega_n} \P_n(\sigma) \P_n(\tau)
 \exp\left(\frac 1{2d} \sum_{ijk\ell} X_{ij} X_{k\ell} A_{ik} A_{j\ell} + \nu_1 + \nu_2 + \xi_n\right),
\end{align*}
where
\begin{align*}
  \nu_1 &= -\frac{1}{2d} \sum_{ij} A_{ii} A_{jj} \pi_i \pi_j, \\
  \nu_2 &= -\frac{1}{2d^2} \sum_{ijk\ell} A_{ik}^2 A_{j\ell}^2 \pi_i \pi_j \pi_k \pi_\ell, \mbox{ and} \\
  \xi_n &= O(n^{-1/2}) \sum_{ij} |X_{ij}| + O(n^{-1}) \left(\sum_{ij} |X_{ij}|\right)^2.
\end{align*}
\end{lemma}

\begin{proof}
 Define 
 \[
  W_{uv}(G, \sigma)
  = \begin{cases}
     \frac{M_{\sigma_u,\sigma_v}}{d} &\text{if $(u, v) \in E(G)$} \\
     \frac{1 - \frac{M_{\sigma_u,\sigma_v}}{n}}{1 - \frac dn} &\text{if $(u, v) \not \in E(G)$.}
    \end{cases}
 \]
 Then we may write out
 \begin{align*}
  Y_n &= \sum_{\sigma \in \Omega_n} \frac{\P_n(G, \sigma)}{\Q_n(G)} \\
  &= \sum_{\sigma \in \Omega_n} \P_n(\sigma)
  \prod_{u, v} W_{uv}(G, \sigma).
 \end{align*}
 Squaring both sides and taking expectations,
 \begin{align}
  \E_{\Q_n} Y_n^2
  &= \E_{\Q_n} \sum_{\sigma,\tau \in \Omega_n} \P_n(\sigma) \P_n(\tau) \prod_{u,v} W_{uv}(G, \sigma) W_{uv}(G, \tau) \notag \\
  &= \sum_{\sigma,\tau \in \Omega_n} \P_n(\sigma) \P_n(\tau) \prod_{u,v} \E_{\Q_n} [W_{uv}(G, \sigma) W_{uv}(G, \tau)],
  \label{eq:second-moment-1}
 \end{align}
 where the last equality holds because under $\Q_n$, and for any fixed $\sigma$,
 the variables $W_{uv}(G,\sigma)$ are independent as $u$ and $v$ vary.

Let us compute the inner expectation in~\eqref{eq:second-moment-1}.
Recall that under $\Q_n$, $(u, v) \in E(G)$ with probability $\frac dn$.
Writing (for brevity) $s$ for $M_{\sigma_u \sigma_v}$ and $t$ for $M_{\tau_u \tau_v}$, we have
\begin{align*}
 \E_{\Q_n} W_{uv}(G, \sigma) W_{uv}(G, \tau)
 &= \frac{st}{d^2} \cdot \frac dn + \frac{(1 - \frac sn)(1 - \frac tn)}{(1 - \frac dn)^2} (1 - \frac dn) \\
 &= \frac{st}{nd} + \left(1 - \frac sn\right)\left(1 - \frac tn\right)\left(1 + \frac dn + \frac{d^2}{n^2} + O(n^{-3})\right) \\
 &= 1 + \frac{(s - d)(t - d)}{nd} + \frac{(s-d)(t-d)}{n^2} + O(n^{-3})
\end{align*}
Setting $q = (s-d)(t-d)$, and using the fact that
$1 + x = \exp(x - x^2/2 + O(x^3))$, we have
\[
 \E_{\Q_n} W_{uv}(G, \sigma) W_{uv}(G, \tau)
 = \exp\left(
 \frac{q}{dn} + \frac{q}{n^2} - \frac{q^2}{2d^2 n^2} + O(n^{-3})
 \right).
\]
Now, if $(\sigma_u, \tau_u, \sigma_v, \tau_v) = (i, j, k, \ell)$
then $(s-d)(t-d) = (M_{ik} - d)(M_{j\ell} - d) = A_{ik} A_{j\ell}$. Hence,
\begin{equation}\label{eq:second-moment-2}
 \E_{\Q_n} W_{uv}(G, \sigma) W_{uv}(G, \tau)
 = \exp\left(
 \frac{A_{ik} A_{j\ell}}{dn} + \frac{A_{ik} A_{j\ell} }{n^2} - \frac{(A_{ik} A_{j\ell})^2}{2d^2 n^2} + O(n^{-3})
 \right).
\end{equation}
Let $N_{ijk\ell} = |\{\{u, v\}: \sigma_u = i, \tau_u = j, \sigma_v = k, \tau_v = \ell\}|$.
Plugging~\eqref{eq:second-moment-2} into~\eqref{eq:second-moment-1}, we have
\begin{align}
  \E_{\Q_n} Y_n^2
  &= (1 + O(n^{-1})) \sum_{\sigma,\tau \in \Omega_n} \P_n(\sigma) \P_n(\tau) \exp\left(
    \sum_{ijk\ell=1}^s N_{ijk\ell} \left(
 \frac{A_{ik} A_{j\ell}}{dn} + \frac{A_{ik} A_{j\ell} }{n^2} - \frac{(A_{ik} A_{j\ell})^2}{2d^2 n^2}
 \right)
  \right) \label{eq:second-moment-3}
\end{align}
where the $(1 + O(n^{-1}))$ term arises because $\sum_{ijk\ell} N_{ijk\ell} \le n^2$.
Applying Lemma~\ref{lem:Nijkl-Xijkl} (below) now finishes the proof.
\end{proof}

The last step in the proof of Lemma~\ref{lem:second-moment-simpl} requires us to replace
$N_{ijk\ell}$ by its normalized version, $X_{ij}$, and then rearrange the sums in~\eqref{eq:second-moment-3}.
We will do this step in slightly more generality, where we allow $N_{ijk\ell}$ to be defined
on a subset of the vertices. For the purposes of this section it suffices to consider $S = [n]$,
but the general form will be useful when we consider non-reconstruction.

\begin{lemma}\label{lem:Nijkl-Xijkl}
Let $S \subseteq [n]$ such that $\abs{S} = n - o(n)$. Further, let
\begin{align*}
N_{ijk\ell} &\defas N_{ijk\ell}(\sigma,\tau) \defas \abs{\{\{u, v\} : u,v \in S, \sigma_u = i, \tau_u = j, \sigma_v = k, \tau_v = \ell\}}, \\
N_{ij} &\defas N_{ij}(\sigma,\tau) \defas \abs{\{u : u \in S, \sigma_u = i, \tau_u = j\}} \mbox{ and,}\\
X_{ij} &\defas X_{ij}(\sigma, \tau) \defas n^{-1/2}\left(N_{ij} - n \pi_i \pi_j\right) \\
t_{ijk\ell} &\defas \frac{A_{ik}A_{j\ell}}{dn} + \frac{A_{ik} A_{j\ell}}{n^2} - \frac{(A_{ik} A_{j\ell})^2}{2 d^2 n^2}
\end{align*}
Then, we have:
\begin{align*}
    \sum_{ijk\ell} N_{ijk\ell} t_{ijk\ell} = 
\frac 1{2d} \sum_{ijk\ell} X_{ij} X_{k\ell} A_{ik} A_{j\ell} + \nu_1 + \nu_2 + \xi_n,
\end{align*}
where
\begin{align*}
  \nu_1 &= -\frac{1}{2d} \sum_{ij} A_{ii} A_{jj} \pi_i \pi_j, \\
  \nu_2 &= -\frac{1}{4d^2} \sum_{ijk\ell} A_{ik}^2 A_{j\ell}^2 \pi_i \pi_j \pi_k \pi_\ell, \mbox{ and} \\
  \xi_n &= O(n^{-1/2}) \sum_{ij} |X_{ij}| + O(n^{-1}) \left(\sum_{ij} |X_{ij}|\right)^2 + O(n^{-1}).
\end{align*}
\end{lemma}

\begin{proof}
We see that $N_{ijk\ell} = \frac{1}{2} N_{ij} N_{k\ell}$ unless $i = k$ and $j = \ell$, 
in which case $N_{ijk\ell} = \binom{N_{ij}}{2} = \frac 12 N_{ij} N_{k\ell} - \frac 12 N_{ij}$.
So, we have
\begin{equation}
    \sum_{ijk\ell} N_{ijk\ell} t_{ijk\ell}
  = \frac 12 \sum_{ijk\ell} N_{ij} N_{k\ell} t_{ijk\ell} - \frac 12 \sum_{ij} N_{ij} t_{ijij}
\label{eq:second-moment-4}
\end{equation}

Recall that $\sum_i \pi_i M_{ik} = d$ for any
fixed $k$ and $\sum_k \pi_k M_{ik} = d$ for any fixed $i$. It follows that
$\sum_i \pi_i A_{ij} = \sum_j \pi_j A_{ij} = 0$.
Hence,
\[
 \sum_i \pi_i t_{ijk\ell} = - \sum_i \pi_i \frac{(A_{ik} A_{j\ell})^2}{2 d^2 n^2}.
\]
Writing $N_{ij} = \sqrt n X_{ij} + n \pi_i \pi_j$, we have
\begin{align*}
 \sum_{ijk\ell} N_{ij} N_{k\ell} t_{ijk\ell}
 &= n \sum_{ijk\ell} X_{ij} X_{k\ell} t_{ijk\ell}
 - \sum_{ijk\ell} \frac{(A_{ik} A_{j\ell})^2}{2 d^2 n^2} \left(n^{3/2} X_{ij} \pi_k \pi_\ell + n^{3/2} X_{k\ell} \pi_i \pi_j + n^2 \pi_i \pi_j \pi_k \pi_\ell\right) \\
 &= n \sum_{ijk\ell} X_{ij} X_{k\ell} t_{ijk\ell}
 - \sum_{ijk\ell} \frac{(A_{ik} A_{j\ell})^2}{2 d^2} \pi_i \pi_j \pi_k \pi_\ell
 + O(n^{-1/2}) \sum_{ij} |X_{ij}|,
\end{align*}
Next, note that
$t_{ijk\ell} = \frac{1}{dn} A_{ik} A_{j\ell} + O(n^{-2})$, and so
\begin{multline*}
 \sum_{ijk\ell} N_{ij} N_{k\ell} t_{ijk\ell}
 = \frac{1}{d} \sum_{ijk\ell} X_{ij} X_{k\ell} A_{ik} A_{j\ell}
 - \frac{1}{2 d^2} \sum_{ijk\ell} (A_{ik} A_{j\ell})^2 \pi_i \pi_j \pi_k \pi_\ell \\
 + O(n^{-1/2}) \sum_{ij} |X_{ij}|
 + O(n^{-1}) \left(\sum_{ij} |X_{ij}|\right)^2;
\end{multline*}
we recognize the second term as $2 \nu_2$, and the last two terms as being part of $\xi_n$.
This takes care of first term in~\eqref{eq:second-moment-4}; for the
second term,
\[
 \sum_{ij} N_{ij} t_{ijij}
 = \sqrt n \sum_{ij} X_{ij} t_{ijij} + n \sum_{ij} \pi_i \pi_j t_{ijij}
 = O(n^{-1/2}) \sum_{ij} |X_{ij}| + \frac{1}{d} \sum_{ij} A_{ii} A_{jj} \pi_i \pi_j + O(n^{-1});
\]
here, the second term is $2 \nu_1$ and the others are part of $\xi_n$.
\end{proof}

The following lemma gives a simpler form for $\nu_1$ and $\nu_2$ appearing above.
We define $B \defas \frac{1}{d} \diag(\pi) A = T - \pi \otimes \1$.
In particular, this will allow us to relate $\nu_1$ and $\nu_2$ to the eigenvalues of $T$.
\begin{lemma}\label{lem:nu}
Let $\nu_1$ and $\nu_2$ be as in Lemma~\ref{lem:Nijkl-Xijkl}. Then, we have:
 \begin{align*}
  \nu_1 &= -\frac d2 \tr(B)^2 \\
  \nu_2 &= -\frac {d^2}{4} \tr(B^2)^2.
 \end{align*}
\end{lemma}

\begin{proof}
  Note that $A_{ii} \pi_i = d B_{ii}$. Hence,
  \[
   \nu_1 = -\frac 1{2d} \sum_{ij} A_{ii} A_{jj} \pi_i \pi_j = -\frac d2 \sum_{ij} B_{ii} B_{jj}
   = -\frac d2 \tr(B)^2.
  \]
  Similarly, since $A_{ik} \pi_i = B_{ik}$ and $A_{ik} \pi_k = A_{ki} \pi_k = B_{ki}$,
  \[
   \nu_2 = -\frac{d^2}{4} \sum_{ijk\ell} B_{ik} B_{ki} B_{j\ell} B_{\ell j}
   = -\frac{d^2}{4} \tr\big((B^{\otimes 2})^2\big)
   = -\frac{d^2}{4} \tr(B^2)^2.
   \qedhere
  \]
\end{proof}

The following lemma shows that $\xi_n$ in Lemma~\ref{lem:Nijkl-Xijkl} is very small in an appropriate sense.
\begin{lemma}\label{lem:xi}
Let $\xi_n$ be as in Lemma~\ref{lem:Nijkl-Xijkl}. If $a_n = o(n^{1/2})$ then $\E \exp(a_n \xi_n) \to 1$.
\end{lemma}

\begin{proof}
 By the central limit theorem, each $X_{ij}$ has a limit in distribution as $n \to \infty$;
 hence $a_n \xi_n \to 0$ in probability. It is therefore enough to show that the sequence
 $\exp(a_n \xi_n)$ is uniformly integrable, but this follows from Hoeffding's inequality.
\end{proof}

We now state the following three results before we prove the main result of this section.
The following proposition characterizes when the exponential of a quadratic form of a sequence of multinomial random variables is uniformly integrable.
Its proof can be found in Section~\ref{sec:UI-multinomials}.
 \begin{proposition}\label{prop:ui}
  Define $X_{ij}$ as in Lemma~\ref{lem:xi}. Then
  \[
  \exp\left(\frac{1}{2d} \sum X_{ij} X_{k\ell} A_{ik} A_{j\ell}\right)
  \]
  is uniformly integrable if $Q(\pi, A/\sqrt{2d}) < 1$, and
  fails to be uniformly integrable if $Q(\pi, A/\sqrt{2d}) > 1$.
 \end{proposition}
 
Using H\"older's inequality, it is fairly straightforward to introduce the $\xi_n$ term:

\begin{lemma}\label{lem:UI-final}
  Define $X_{ij}$ as in Lemma~\ref{lem:xi}. Then
  \[
  \exp\left(\frac{1}{2d} \sum X_{ij} X_{k\ell} A_{ik} A_{j\ell} + \xi_n\right)
  \]
  is uniformly integrable if $Q(\pi, A/\sqrt{2d}) < 1$, and
  fails to be uniformly integrable if $Q(\pi, A/\sqrt{2d}) > 1$.
\end{lemma}
\begin{proof}
Supposing that $Q(\pi, A/\sqrt{2d}) < 1$, we find some $\epsilon > 0$ such that
$Q(\pi, \sqrt{1+\epsilon} A / \sqrt{2d}) < 1$.
Set $a_n = n^{1/3}$ and $b_n = \frac{a_n}{a_n - 1}$ to be the H\"older conjugate of $a_n$. Setting
\begin{align}
W = \vec(X),
\label{eqn:defn-W}
\end{align}
H\"older's inequality and Lemma~\ref{lem:xi} give
\begin{align*}
  &\E_{\sigma,\tau} \exp\left((1+\frac{\epsilon}{2})\left(\frac 1{2d} \sum_{ijk\ell} X_{ij} X_{k\ell} A_{ik} A_{j\ell} + \xi_n\right)\right)\\
 &\le
 \left(\E_{\sigma,\tau} \exp\left(\frac{(1+\frac{\epsilon}{2})b_n}{2d} W^T (A^{\otimes 2}) W\right)\right)^{1/b_n}
 \left(\E \exp((1+\frac{\epsilon}{2})a_n \xi_n)\right)^{1/a_n} \\
 &\le
 \left(\E_{\sigma,\tau} \exp\left(\frac{(1+\frac{\epsilon}{2})b_n}{2d} W^T (A^{\otimes 2}) W\right)\right)^{1/b_n}.
\end{align*}
To check uniform integrability, we apply Proposition~\ref{prop:ui}. For sufficiently large $n$,
we have $b_n \le \frac{1 + \epsilon}{1+\frac{\epsilon}{2}}$ and
\[
 \exp\left(\frac {(1+\frac{\epsilon}{2})b_n}{2d} W^T A^{\otimes 2} W\right)
 \le \max\left\{1, \exp\left(\frac {(1 + \epsilon)}{2d} W^T A^{\otimes 2} W\right)\right\}.
\]
We see from the fact that $Q(\pi, \sqrt{1+\epsilon} A / \sqrt{2d}) < 1$ and
Proposition~\ref{prop:ui} that the right hand side above has a finite expectation.

To summarize, we have shown that if $Z = \exp(\frac{1}{2d} \sum X_{ij} X_{k\ell} A_{ik} A_{j\ell} + \xi_n)$ then
$\E Z^{(1+\epsilon/2)} < \infty$ for some $\epsilon > 0$. It follows that $Z$ is uniformly integrable, as claimed.

To show that $Q(\pi, A/\sqrt{2d}) > 1$ implies non-uniform integrability, requires an almost identical argument,
but using the reverse H\"older inequality instead of the usual H\"older inequality. We omit the details.
\end{proof}

The following lemma calculates the expected value of the exponential of a quadratic form of a Gaussian random vector.
\begin{lemma}\label{lem:gaussian-quadratic-limit}
  Take $Z \sim \normal(0, \Sigma)$, where $\Sigma = \diag(\pi)^{\otimes 2} - \pi^{\otimes 4}$.
  Recall that $\lambda_i$ denote the eigenvalues of $T$, with $1=\lambda_1 \geq \abs{\lambda_2} \geq \cdots \geq \abs{\lambda_s}$.
  If $d \lambda_2^2 < 1$ then
  \[
   \E \exp\left(\frac{1}{2d} Z^T A^{\otimes 2} Z\right) = \prod_{i,j=2}^s \frac{1}{\sqrt{1-d\lambda_i\lambda_j}}.
  \]
  Otherwise, $\E \exp\left(\frac{1}{2d} Z^T A^{\otimes 2} Z\right) = \infty$.
\end{lemma}
\begin{proof}
A standard computation (see, e.g.~\cite{MathaiProvost:92}) shows that if $\mu_1, \dots, \mu_k$ denote the
eigenvalues of $\Sigma \tilde A $ then $\E \exp(Z^T \tilde A Z / 2) = \prod_i \frac{1}{\sqrt{1-\mu_k}}$.
Now,
\[
\Sigma A^{\otimes 2} = (\diag(\pi)^{\otimes 2} - \pi^{\otimes 4}) A^{\otimes 2} 
= (\diag(\pi) A)^{\otimes 2} - (\pi \pi^T A)^{\otimes 2}.
\]
Recall, however, that $A \pi = 0$. Hence, we are interested in the eigenvalues of
$(\diag(\pi) A)^{\otimes 2} = (dB)^{\otimes 2}$. Since the top eigenvalue of $T$ is 1 (with $1$
as its right-eigenvector and $\pi$ as its left-eigenvector), we see that
if $\lambda_1,\cdots,\lambda_s$ are the eigenvalues of $T$ with $\lambda_1 = 1$, then
\[
 \{d \lambda_i \lambda_j: i, j = 2, \dots, s\}
\]
are the eigenvalues of $\frac 1d \Sigma (A \otimes A)$.
\end{proof}

\begin{proof}[Proof of Theorem~\ref{thm:second-moment}]
First of all, note that
\[
 \frac{d\tilde \P_n(G, \sigma)}{d\Q_n} = \frac{Y_n}{\P_n(\Omega_n)} = (1 + o(1)) Y_n.
\]
Hence, it suffices to compute the limit of $\E_{\Q_n} Y_n^2$.

From Lemma~\ref{lem:second-moment-simpl}, we see that we need to calculate the limit of the quantity
\begin{align*}
  \E_{\sigma,\tau \in \Omega_n} \exp\left(\frac 1{2d} \sum_{ijk\ell} X_{ij} X_{k\ell} A_{ik} A_{j\ell} + \xi_n\right).
\end{align*}
Lemma~\ref{lem:UI-final} establishes that the above sequence is uniformly integrable.

Now, note that $(N_{ij})_{i,j=1}^s$ is distributed as a multinomial random vector with $n$ trials and probabilities
$\pi_i \pi_j$. In particular, $\E N_{ij} = \pi_i \pi_j$, $\Var(N_{ij}) = \pi_i \pi_j - (\pi_i \pi_j)^2$, and
$\Cov(N_{ij} N_{k\ell}) = -\pi_i \pi_j \pi_k \pi_\ell$ if $\{i,j\} \ne \{k,\ell\}$.
Since $X_{ij} = n^{-\frac{1}{2}}\left(N_{ij} - n \pi_i \pi_j\right)$, central limit theorem implies that $W$ converges in distribution to a Gaussian random vector,
$Z$ with mean $0$ and covariance matrix $\diag(\pi)^{\otimes 2} - \pi^{\otimes 4}$.
Using Lemma~\ref{lem:gaussian-quadratic-limit} shows us that
\begin{equation}\label{eq:second-moment-with-nu}
 \E_{\Q_n} Y_n^2 \to \exp(\nu_1 + \nu_2) \prod_{i,j=2}^s \frac{1}{\sqrt{1-d\lambda_i \lambda_j}}.
\end{equation}
Going back to Lemma~\ref{lem:nu}, we have
\[
 \nu_1 = -\frac{d}{2} \tr(B)^2 = -\frac 12 \sum_{i,j=2}^s d\lambda_i \lambda_j
\]
and
\[
 \nu_2 = -\frac{d^2}{4} \tr(B^2)^2 = -\frac 14 \sum_{i,j=2}^2 (d \lambda_i \lambda_j)^2.
\]
Hence, the right hand side of~\eqref{eq:second-moment-with-nu} is equal to
\[
 \prod_{i,j} \psi(d\lambda_i\lambda_j),
\]
as claimed.
\end{proof}

\section{Non-reconstructability}

%
In this section, we prove Theorem~\ref{thm:non-reconstruct}.
The following proposition is the main technical result that we use to prove Theorem~\ref{thm:non-reconstruct}.
It shows that under the uniform integrability condition, for any two fixed configurations on a finite set of nodes,
the total variation distance between the distribution on graphs conditioned on these two configurations respectively goes to zero.
\begin{proposition}\label{prop:TVdistance}
Suppose $Q\left(\pi,A/\sqrt{2d}\right) < 1$. Then, for any fixed $r > 0$, and for any two configurations $\left(a_1,a_2,\cdots, a_r\right)$ and
$\left(b_1,b_2,\cdots, b_r\right)$, we have:
\begin{align*}
  TV\left(\prob{G \middle\vert \sigma_u=a_u \mbox{ for } u\in[r]},
\prob{G \middle\vert \sigma_u=b_u \mbox{ for } u\in[r]}\right) = o(1),
\end{align*}
where $TV(\mathbb{P}_1,\mathbb{P}_2)$ denotes the total variation distance between the two distributions $\mathbb{P}_1$ and $\mathbb{P}_2$.
\end{proposition}
\begin{proof}
We will first prove the statement of the proposition with $\P_n$ replaced by
$\hat \P_n = (\P_n \mid \Omega_n)$; i.e., the
block model conditioned on having almost the expected label frequencies.

We start by using the definition of total variation distance:
\begin{align*}
& \quad TV\left(\probcond{G \middle\vert \sigma_u=a_u \mbox{ for } u\in[r]},
\probcond{G \middle\vert \sigma_u=b_u \mbox{ for } u\in[r]}\right) \\
&= \sum_G \abs{\probcond{G \middle\vert \sigma_u=a_u \mbox{ for } u\in[r]} -
\probcond{G \middle\vert \sigma_u=b_u \mbox{ for } u\in[r]}} \\
&= \sum_G \abs{\probcond{G \middle\vert \sigma_u=a_u \mbox{ for } u\in[r]} -
\probcond{G \middle\vert \sigma_u=b_u \mbox{ for } u\in[r]}} \frac{\sqrt{\probER{G}}}{\sqrt{\probER{G}}}\\
&\stackrel{(a)}{\leq} \left(\sum_G \probER{G}\right)^{1/2} \left(\sum_G \frac{\left(\probcond{G \middle\vert \sigma_u=a_u \mbox{ for } u\in[r]} -
\probcond{G \middle\vert \sigma_u=b_u \mbox{ for } u\in[r]}\right)^2}{\probER{G}}\right)^{1/2}\\
&= \left(\sum_G \frac{\left(\sum_{\sigtil} \probcond{\sigtil}\left(\probcond{G \middle\vert a,\sigtil} -
\probcond{G \middle\vert b,\sigtil}\right)\right)^2}{\probER{G}}\right)^{1/2},
\end{align*}
where $(a)$ follows from Cauchy-Schwartz inequality and $\sigtil$ denotes an assignment on $[n]\setminus [r]$.
We can expand the numerator as follows:
\begin{align*}
&\left(\sum_{\sigtil} \probcond{\sigtil}\left(\probcond{G \middle\vert a,\sigtil} -
\probcond{G \middle\vert b,\sigtil}\right)\right)^2 \\
&= \sum_{\sigtil,\tautil} \probcond{\sigtil}\probcond{\tautil} \left(
\probcond{G \middle\vert a,\sigtil}\probcond{G \middle\vert a,\tautil} + \probcond{G \middle\vert b,\sigtil}\probcond{G \middle\vert b,\tautil} \right.\\
&\qquad\qquad\qquad\qquad \left.- \probcond{G \middle\vert a,\sigtil}\probcond{G \middle\vert b,\tautil} - \probcond{G \middle\vert b,\sigtil}\probcond{G \middle\vert a,\tautil}
\right).
\end{align*}
We will now show that the value of
\begin{align*}
\sum_{\sigtil,\tautil} \probcond{\sigtil}\probcond{\tautil} \sum_G \frac{\probcond{G \middle\vert a,\sigtil}\probcond{G \middle\vert b,\tautil}}{\probER{G}},
\end{align*}
is independent of $a$ and $b$ (upto $o(1)$). This will prove our claim.
 Define
 \[
  W_{uv}(G, \sigma)
  \defas \begin{cases}
     \frac{M_{\sigma_u,\sigma_v}}{d} &\text{if $(u, v) \in E(G)$,} \\
     \frac{1 - \frac{M_{\sigma_u,\sigma_v}}{n}}{1 - \frac dn} &\text{if $(u, v) \not \in E(G)$,}
    \end{cases}
 \]
and let $q_{ijk\ell} = (M_{ik} - d)(M_{j\ell} - d)/n = A_{ik} A_{j\ell} / n$,
and $t_{ijk\ell} = \frac {q_{ijk\ell}} d + \frac {q_{ijk\ell}} n - \frac{q_{ijk\ell}^2}{2d^2}$.
We have:
\begin{align}
&\sum_{\sigtil,\tautil} \probcond{\sigtil}\probcond{\tautil} \sum_G \frac{\probcond{G \middle\vert a,\sigtil}\probcond{G \middle\vert b,\tautil}}{\probER{G}}\nonumber\\
&=\sum_{\sigtil,\tautil} \probcond{\sigtil}\probcond{\tautil}\prod_{u,v \in [n]} \E_{\Q_n} [W_{uv}(G, a,\sigtil) W_{uv}(G, b,\tautil)] \nonumber\\
&= \hat\E_{\sigtil,\tautil} \prod_{u,v \in [n]\setminus [r]} \left(1+
t_{\sigtil_u\tautil_u\sigtil_v\tautil_v}+\bigoh{\frac{1}{n^3}}\right)
\prod_{\stackrel{u\in[r]}{v \in [n]\setminus [r]}} \left(1+
t_{a_u b_u\sigtil_v\tautil_v}+\bigoh{\frac{1}{n^3}}\right)
\prod_{u,v\in[r]} \left(1+
t_{a_u b_u a_v b_v}+\bigoh{\frac{1}{n^3}}\right)\nonumber\\
&= \hat\E_{\sigtil,\tautil} \prod_{i,j,k,\ell \in [s]} \left(1+
t_{ijk\ell}+\bigoh{\frac{1}{n^3}}\right)^{\Ntil_{ijk\ell}} 
\prod_{\stackrel{u\in[r]}{i,j \in [s]}} \left(1+
t_{a_u b_u ij}+\bigoh{\frac{1}{n^3}}\right)^{\Ntil_{ij}}
\prod_{u,v\in[r]} \left(1+
t_{a_u b_u a_v b_v}+\bigoh{\frac{1}{n^3}}\right),\label{eqn:main}
\end{align}
where $\Ntil_{ijk\ell} = \left|\{\{u, v\}: \sigtil_u = i, \tautil_u = j, \sigtil_v = k, \tautil_v = \ell\}\right|$, and
$\Ntil_{ij} = \left|\set{v: \sigtil_v=i, \tautil_v=j}\right|$.
We first note that the last term in \eqref{eqn:main} can be simplified as follows:
\begin{align*}
&\prod_{u,v\in[r]} \left(1+t_{a_u b_u a_v b_v}+\bigoh{\frac{1}{n^3}}\right)
= \prod_{u,v\in[r]} \left(1+\bigoh{\frac{1}{n}}\right)
= \left(1+\bigoh{\frac{1}{n}}\right)^{r^2} = 1+\bigoh{\frac{1}{n}}.
\end{align*}
For the second term in \eqref{eqn:main}, we have:
\begin{align}
\prod_{{i,j \in [s]}} \left(1+
t_{a_u b_u ij}+\bigoh{\frac{1}{n^3}}\right)^{\Ntil_{ij}}
&=\prod_{i,j \in [s]} \left(1+
\frac{q_{a_u b_u ij}}{d}+\bigoh{\frac{1}{n^2}}\right)^{\Ntil_{ij}} \nonumber\\
&\stackrel{(\zeta_1)}{=}\prod_{i,j \in [s]} \exp\left(\Ntil_{ij}\left(\frac{q_{a_u b_u ij}}{d}+\bigoh{\frac{1}{n^2}}\right)\right) \nonumber\\
&\stackrel{(\zeta_2)}{=} (1+o(1)) \prod_{i,j \in [s]} \exp\left(\frac{q_{a_u b_u ij}}{d} \cdot \Ntil_{ij} \right) \nonumber\\
&= (1+o(1))\prod_{i,j \in [s]} \exp\left(\frac{nq_{a_u b_u ij}}{d} \cdot \left(\frac{\Ntil_{ij}}{n}-\pi_i \pi_j\right) \right) \cdot
	\exp\left(\frac{nq_{a_u b_u ij}}{d} \cdot \pi_i \pi_j \right), \label{eqn:term2}
\end{align}
where $(\zeta_1)$ follows from the fact that $1+x = \exp\left(x + \bigoh{x^2}\right)$ and $(\zeta_2)$ follows from the fact that $\Ntil_{ij} < n$.
We first note that:
\begin{align*}
\prod_{i,j \in [s]} \exp\left(\frac{nq_{a_u b_u ij}}{d} \cdot \pi_i \pi_j \right)
&= \prod_{i,j \in [s]} \exp\left(\frac{\pi_i \pi_j A_{a_u i}A_{b_u j}}{d}\right) \\
&= \exp\left(\sum_{i,j \in [s]} \frac{\pi_i \pi_j A_{a_u i}A_{b_u j}}{d}\right) \\
&= \exp\left( \frac{\left(\sum_{i\in [s]}\pi_i A_{a_u i}\right)\left(\sum_{j\in [s]}\pi_j A_{b_u j}\right)}{d}\right)
= 1.
\end{align*}
Looking now at the first two terms of \eqref{eqn:main} we obtain (using \eqref{eqn:term2}):
\begin{align*}
&\hat\E_{\sigtil,\tautil} \prod_{i,j,k,l \in [s]} \left(1+
t_{ijk\ell}+\bigoh{\frac{1}{n^3}}\right)^{\Ntil_{ijk\ell}}
\prod_{\stackrel{u \in [r]}{i,j \in [s]}} \exp\left(\frac{nq_{a_u b_u ij}}{d} \cdot \left(\frac{\Ntil_{ij}}{n}-\pi_i \pi_j\right) \right) \\
&\stackrel{(\zeta_1)}{=} (1+o(1)) \hat\E_{\sigtil,\tautil} \exp\left(
  \sum_{ijk\ell} \Ntil_{ijk\ell} t_{ijk\ell}
\right)
\prod_{\stackrel{u \in [r]}{i,j \in [s]}} \exp\left(\frac{nq_{a_u b_u ij}}{d} \cdot \frac{\Xtil_{ij}}{\sqrt{n}} \right) \\
&\stackrel{(\zeta_2)}{=}
(1+o(1))\hat\E_{\sigtil,\tautil}
\exp\left(\frac 1{2d} \sum_{ijk\ell} \Xtil_{ij} \Xtil_{k\ell} A_{ik} A_{j\ell} + \nu_1 + \nu_2 + \xitil_n\right)
\prod_{\stackrel{u \in [r]}{i,j \in [s]}} \exp\left(\frac{nq_{a_u b_u ij}}{d} \cdot \frac{\Xtil_{ij}}{\sqrt{n}} \right) \\
&= (1+o(1))\exp\left(\nu_1+\nu_2\right)
\hat\E_{\sigtil,\tautil}
\exp\left(\frac 1{2d} \sum_{ijk\ell} \Xtil_{ij} \Xtil_{k\ell} A_{ik} A_{j\ell} + \xitil_n\right)
\prod_{\stackrel{u \in [r]}{i,j \in [s]}} \exp\left(\frac{nq_{a_u b_u ij}}{d} \cdot \frac{\Xtil_{ij}}{\sqrt{n}} \right) \\
\end{align*}
where $\Xtil_{ij}\defas n^{-1/2}\left(\Ntil_{ij}-n \pi_i\pi_j\right)$, $(\zeta_1)$ follows from the fact that $1+x = \exp\left(x + \bigoh{x^2}\right)$ and since $\Ntil_{ijk\ell}< n^2$, $(\zeta_2)$ follows from Lemma~\ref{lem:Nijkl-Xijkl}.
Note that $\exp\left(\frac 1{2d} \sum_{ijk\ell} \Xtil_{ij} \Xtil_{k\ell} A_{ik} A_{j\ell} + \xitil_n\right)$ is
independent of $a$ and $b$ and from Lemma~\ref{lem:UI-final}, we also know that it is uniformly integrable.
On the other hand, since $\abs{\Xtil_{ij}} \leq \sqrt{n}$, we see that $\exp\left(\sum_{u \in [r],i,j \in [s]} \frac{nq_{a_u b_u ij}}{d}
\cdot \frac{\Xtil_{ij}}{\sqrt{n}}\right)$
is uniformly bounded and hence is uniformly integrable. Moreover, $\Xtil_{ij} \rightarrow \Normal\left(0,\pi_i\pi_j - (\pi_i\pi_j)^2\right)$.
So we see that,
\begin{align*}
\hat\E_{\sigtil,\tautil} {\exp\left(\frac 1{2d} \sum_{ijk\ell} \Xtil_{ij} \Xtil_{k\ell} A_{ik} A_{j\ell} + \xitil_n\right) \cdot \prod_{\stackrel{u \in [r]}{i,j \in [s]}} \exp\left(\frac{n q_{a_u b_u ij}}{d} \cdot \frac{\Xtil_{ij}}{\sqrt{n}} \right)}
\end{align*}
converges to a finite quantity that is independent of $a$ and $b$.
This proves the statement of the proposition with $\P_n$ replaced by
$\hat \P_n = (\P_n \mid \Omega_n)$.
Noting that
\begin{align*}
TV\left(\prob{G \middle\vert \sigma_u=a_u \mbox{ for } u\in[r]},
\probcond{G \middle\vert \sigma_u=a_u \mbox{ for } u\in[r]}\right) = o(1), \; \forall \; a
\end{align*}
gives us the desired result.
\end{proof}

In order to prove Theorem~\ref{thm:non-reconstruct}, we use the following lemma which is an easy consequence of
Proposition~\ref{prop:TVdistance}.
\begin{lemma}\label{lem:condTV}
Suppose $Q\left(\pi,A/\sqrt{2d}\right) < 1$. Then, for any set $S$ such that $|S|$ is a constant, $u \notin S$, we have:
\begin{align*}
  \expec{TV\left(\prob{\sigma_u \middle\vert G, \sigma_S}, \pi \right) \middle\vert \sigma_S} = o(1).
\end{align*}
\end{lemma}
\begin{proof}
\begin{align*}
  \expec{TV\left(\prob{\sigma_u \middle\vert G, \sigma_S}, \pi \right) \middle\vert \sigma_S}
  &= \sum_{\sigma_u} \prob{\sigma_u} \sum_G \abs{\frac{\prob{G \middle\vert \sigma_u, \sigma_S}}{\prob{G \middle\vert \sigma_S}} - 1}
	\prob{G \middle\vert \sigma_S} \\
  &= \sum_{i} \pi(i) TV\left(\prob{G \middle\vert \sigma_u=i, \sigma_S},\prob{G \middle\vert \sigma_S}\right) = o(1),
\end{align*}
where the last step follows from Proposition~\ref{prop:TVdistance}.
\end{proof}

We are now ready to prove Theorem~\ref{thm:non-reconstruct}.
\begin{proof}[Proof of Theorem~\ref{thm:non-reconstruct}]
We will show that $\lim_{n \to \infty} \expec{\overlap(\calA(G), \sigma)} = 0$. Theorem~\ref{thm:non-reconstruct}
then follows from Markov's inequality. We first bound $\expec{\overlap(\calA(G), \sigma)}$ as follows:
\begin{align}
  \expec{\overlap(\sigma,\calA(G))}
    &= \frac{1}{n}\expec{\max_\rho \sum_{i=1}^s \left(N_{i \rho(i)} (\sigma, \calA(G)) - \frac 1n N_i(\sigma) N_{\rho(i)}(\calA(G))\right)} \nonumber\\
    &\leq  \frac{1}{n}\sum_\rho \expec{\abs{\sum_{i=1}^s \left(N_{i \rho(i)} (\sigma, \calA(G)) - \frac 1n N_i(\sigma) N_{\rho(i)}(\calA(G))\right)}}.
\label{eqn:overlap-bound}
\end{align}
We will now show that each of the terms in the above summation goes to zero. Wlog, let $\rho$ be identity. Fix $i\in [s]$ and
consider the term $\E{\abs{\left(N_{i i} - \frac 1n N_i(\sigma) N_{i}(\calA(G))\right)}}$ (for brevity,
we suppress $\sigma, \calA(G)$ in $N_{i i} (\sigma, \calA(G))$).
Using Jensen's inequality, it is sufficient to bound
\begin{align}
\E{\left(N_{i i} - \frac 1n N_i(\sigma) N_{i}(\calA(G))\right)^2}
&= \expec{N_{ii}^2 - \frac 2n N_{ii}N_i(\sigma) N_{i}(\calA(G)) + \frac{1}{n^2} N_i^2(\sigma) N_{i}^2(\calA(G))}.
\label{eqn:Niibound}
\end{align}
We will now calculate each of the above three terms.
\begin{align}
\E N_{ii}^2 = \E \left(\sum_u \indicator{\sigma_u = i} \indicator{\calA(G)_u = i}\right)^2
&= \sum_{u,v} \E \indicator{\sigma_u = i} \indicator{\calA(G)_u = i} \indicator{\sigma_v = i} \indicator{\calA(G)_v = i} \nonumber\\
& = \sum_{u,v} \E \indicator{\sigma_u = i} \indicator{\calA(G)_u = i} \indicator{\sigma_v = i} \indicator{\calA(G)_v = i} \nonumber\\
& = \sum_{u,v} \expec{ \expec{\indicator{\sigma_u = i} \indicator{\calA(G)_u = i} \indicator{\sigma_v = i} \indicator{\calA(G)_v = i}\middle\vert G}} \nonumber\\
& = \sum_{u,v} \expec{ \expec{\indicator{\sigma_u = i} \indicator{\sigma_v = i} \middle\vert G} \indicator{\calA(G)_u = i} \indicator{\calA(G)_v = i}} \nonumber\\
&= \left(\pi(i)^2 \expec{ \indicator{\calA(G)_u = i} \indicator{\calA(G)_v = i}} + o(1) \right) n^2,
\label{eqn:Niibound1}
\end{align}
where the last step follows from Lemma~\ref{lem:condTV}.
Coming to the second term, we have:
\begin{align}
\E N_{ii} N_i(\sigma) N_{i}(\calA(G))
&= \E \left(\sum_u \indicator{\sigma_u = i} \indicator{\calA(G)_u = i}\right)
	\left(\sum_u \indicator{\sigma_u = i}\right)
	\left(\sum_u \indicator{\calA(G)_u = i}\right) \nonumber\\
&= \sum_{u,v,w} \expec{\expec{ \indicator{\sigma_u = i} \indicator{\calA(G)_u = i}
	\indicator{\sigma_v = i} \indicator{\calA(G)_w = i} \middle\vert G}} \nonumber\\
&= \sum_{u,v,w} \expec{\expec{ \indicator{\sigma_u = i} \indicator{\sigma_v = i} \middle\vert G}
	\indicator{\calA(G)_u = i} \indicator{\calA(G)_w = i}} \nonumber\\
&= \left(\pi(i)^2 \expec{ \indicator{\calA(G)_u = i} \indicator{\calA(G)_v = i}} + o(1) \right) n^3,
\label{eqn:Niibound2}
\end{align}
where the last step again follows from Lemma~\ref{lem:condTV}.
A similar argument shows that
\begin{align}
\E N_i^2(\sigma) N_{i}^2(\calA(G)) = \left(\pi(i)^2 \expec{ \indicator{\calA(G)_u = i} \indicator{\calA(G)_v = i}} + o(1) \right) n^4.
\label{eqn:Niibound3}
\end{align}
Plugging~\eqref{eqn:Niibound1},~\eqref{eqn:Niibound2} and~\eqref{eqn:Niibound3} in~\eqref{eqn:Niibound} shows that
\begin{align*}
\E{\left(N_{i i} - \frac 1n N_i(\sigma) N_{i}(\calA(G))\right)^2} = o(n^2). 
\end{align*}
This finishes the proof.
\end{proof}

\section{Examples}\label{sec:examples}
In this section, we will present a proof of Proposition~\ref{prop:ex1}.
We use the following lemma, which is a restatement of Lemma~4.7 from Mossel et al. \cite{MoNeSl:13}.
It establishes an approximate Markov structure on the labels of two sets of nodes with a small separator.
\begin{lemma}(Restatement of Lemma~4.7 from \cite{MoNeSl:13})\label{lem:separation}
Let $A=A(G)$, $B=B(G)$ and $C=C(G) \subset V$ be a (random) partition of $V$ such that $B$ separates $A$ and $C$ in $G$.
If $\abs{A \cup B} = o(\sqrt{n})$ for a.a.e. $G$, then
\begin{align*}
  \prob{\sigma_A \middle\vert \sigma_{B\cup C}, G} = (1+o(1))\prob{\sigma_A \middle\vert \sigma_{B}, G_{A\cup B}},
\end{align*}
for a.a.e. $G$ and $\sigma$.
\end{lemma}
\noindent\textbf{Remark}: Lemma~4.7 from \cite{MoNeSl:13} only states that
\begin{align*}
  \prob{\sigma_A \middle\vert \sigma_{B\cup C}, G} = (1+o(1))\prob{\sigma_A \middle\vert \sigma_{B}, G}.
\end{align*}
However, its proof directly gives us the stronger statement above.
We are now ready to prove Proposition~\ref{prop:ex1}.
\begin{proof}[Proof of Proposition~\ref{prop:ex1}]
We will prove the proposition by showing that the conclusion of Lemma~\ref{lem:condTV} holds i.e.,
for any set $S$ of constant size and $u\notin S$,
\begin{align}
  \expec{TV\left(\prob{\sigma_u \middle\vert G, \sigma_S}, \pi \right) \middle\vert \sigma_S} = o(1).
\label{eqn:ex1-int}
\end{align}
Since the size of the largest component is $\bigoh{\log n}$, $u$ and $S$
are disconnected a.a.s. Choosing $A$ to be the component of $u$, $B$ to be $\emptyset$ and $C$ to be $V\setminus A$ in
Lemma~\ref{lem:separation}, we see
that $\sigma_A$ and $\sigma_S$ are a.a.s. independent given $G_{A}$. Hence $\sigma_u$ and $\sigma_S$ are also a.a.s. independent
given $G_A$. So, we see that
\begin{align*}
  TV\left(\prob{\sigma_u \middle\vert \sigma_{S}, G}, \prob{\sigma_u \middle\vert G_A}\right) \rightarrow 0 \mbox{ for a.e. } G.
\end{align*}
In order to show \eqref{eqn:ex1-int}, it suffices to show that
\begin{align*}
  TV\left(\prob{\sigma_u \middle\vert G_A}, \pi\right) \rightarrow 0 \mbox{ for a.e. } G.
\end{align*}
This in turn follows if we show that
\begin{align*}
  TV\left(\prob{G_A \middle\vert \sigma_u = 1}, \prob{G_A \middle\vert \sigma_u = 2} \right) \to 0 \text{ for a.e. } G.
\end{align*}
This is clearly true since
\begin{itemize}
  \item	$\P_n \left(G_A \middle\vert \sigma_u = 1\right) = \Q_{pn}\left(G_A\right)$ and
$\P_n \left(G_A \middle\vert \sigma_u = 2\right) = \Q_{(1-p)n}\left(G_A\right)$,
  \item	$\lim_{r \rightarrow \infty} \Q_n \left(\abs{G_A} > r\right) = 0$, and
  \item	$\Q_n\left(G_A \mid \abs{G_A} = r\right)$ converges in distribution for every fixed $r$.
\end{itemize}
This proves~\eqref{eqn:ex1-int}. The rest of the proof is the same as that of Theorem~\ref{thm:non-reconstruct}.
\end{proof}

\section{Open problems}\label{sec:problems}

Our results show that the Kesten-Stigum bound is not the threshold for reconstructability
in the stochastic block model. Indeed, Propositions~\ref{prop:reconstruction-belowKS}
and~\ref{prop:ex1} show that even for the two cluster models with a fixed partition size,
reconstructability does not have a threshold behavior in $\lambda_2^2 d$.
\begin{question}
What is the precise boundary between reconstructability and non-reconstructability in unbalanced two cluster models?
\end{question}

We obtain non-distinguishability and non-reconstructability by showing the existence of a certain second moment.
The existence of any $(1+\epsilon)$ moment suffices for part of our results. However, calculating
such moments seems much harder.
\begin{question}
Characterize the region where the $1+\epsilon$ moment is finite for some $\epsilon < 1$.
\end{question}

The function $Q$ in our uniform integrability condition is not explicit.
Currently, the only way we can estimate $Q$ is via numerical optimization.
\begin{question}
Can we evaluate $Q$ explicitly? Can we obtain good bounds for it?
\end{question}

Finally, we would like to stress that the most novel contribution of our work is to relate non-reconstructability
with non-distinguishability. We believe that this connection might prove useful in other contexts where
non-reconstructability results have so far proved elusive.


\begin{thebibliography}{KMM{\etalchar{+}}13}

\bibitem[BJR07]{BoJaRi:07}
B{\'e}la Bollob{\'a}s, Svante Janson, and Oliver Riordan.
\newblock The phase transition in inhomogeneous random graphs.
\newblock {\em Random Structures \& Algorithms}, 31(1):3--122, 2007.

\bibitem[CK01]{CK:01}
A.~Condon and R.M. Karp.
\newblock Algorithms for graph partitioning on the planted partition model.
\newblock {\em Random Structures and Algorithms}, 18(2):116--140, 2001.

\bibitem[CO10]{CO:10}
A.~Coja-Oghlan.
\newblock Graph partitioning via adaptive spectral techniques.
\newblock {\em Combinatorics, Probability and Computing}, 19(02):227--284,
  2010.

\bibitem[DF89]{DF:89}
M.E. Dyer and A.M. Frieze.
\newblock The solution of some random {NP}-hard problems in polynomial expected
  time.
\newblock {\em Journal of Algorithms}, 10(4):451--489, 1989.

\bibitem[DKMZ11]{Z:2011}
A.~Decelle, F.~Krzakala, C.~Moore, and L.~Zdeborov\'a.
\newblock Asymptotic analysis of the stochastic block model for modular
  networks and its algorithmic applications.
\newblock {\em Physics Review E}, 84:066106, Dec 2011.

\bibitem[HLL83]{HLL:83}
P.W. Holland, K.B. Laskey, and S.~Leinhardt.
\newblock Stochastic blockmodels: First steps.
\newblock {\em Social Networks}, 5(2):109 -- 137, 1983.

\bibitem[Jan95]{Janson:95}
Svante Janson.
\newblock Random regular graphs: asymptotic distributions and contiguity.
\newblock {\em Combinatorics, Probability and Computing}, 4(04):369--405, 1995.

\bibitem[JS98]{JS:98}
M.~Jerrum and G.B. Sorkin.
\newblock The {M}etropolis algorithm for graph bisection.
\newblock {\em Discrete Applied Mathematics}, 82(1-3):155--175, 1998.

\bibitem[KMM{\etalchar{+}}13]{Krzakala_etal:13}
F.~Krzakala, C.~Moore, E.~Mossel, J.~Neeman, A.~Sly, Zdeborova L, and P.~Zhang.
\newblock Spectral redemption: clustering sparse networks.
\newblock arXiv:1306.5550, 2013.

\bibitem[Mas14]{Massoulie:13}
L.~Massouli\'e.
\newblock Community detection thresholds and the weak {R}amanujan property.
\newblock arXiv:1311:3085, 2014.

\bibitem[MNS13]{MoNeSl:13}
E.~Mossel, J.~Neeman, and A.~Sly.
\newblock Stochastic block models and reconstruction.
\newblock arXiv:1202.4124, 2013.

\bibitem[MNS14]{MoNeSl:14}
Elchanan Mossel, Joe Neeman, and Allan Sly.
\newblock A proof of the block model threshold conjecture.
\newblock arXiv:1311.4115, 2014.

\bibitem[MP92]{MathaiProvost:92}
A.M. Mathai and Serge~B. Provost.
\newblock {\em Quadratic Forms in Random Variables}.
\newblock Statistics Series. Taylor \& Francis, 1992.

\bibitem[RW92]{RW:92}
R.W. Robinson and N.C. Wormald.
\newblock Almost all cubic graphs are {H}amiltonian.
\newblock {\em Random Structures and Algorithms}, 3(2):117--125, 1992.

\bibitem[SN97]{SN:97}
T.A.B. Snijders and K.~Nowicki.
\newblock Estimation and prediction for stochastic blockmodels for graphs with
  latent block structure.
\newblock {\em Journal of Classification}, 14(1):75--100, 1997.

\bibitem[VAC14]{VerzelenAriasCastro:14}
Nicolas Verzelen and Ery Arias-Castro.
\newblock Community detection in sparse random networks.
\newblock {\em arXiv preprint arXiv:1308.2955}, 2014.

\bibitem[Wor99]{Wormald}
N.C. Wormald.
\newblock Models of random regular graphs.
\newblock {\em London Mathematical Society Lecture Note Series}, pages
  239--298, 1999.

\end{thebibliography}

\newcommand{\etalchar}[1]{$^{#1}$}

\newpage
\appendix
\section{UI and multinomials}\label{sec:UI-multinomials}

Here, we restate and prove Proposition~\ref{prop:ui}.
Recall that $\Delta_s$ denotes the set $\{(\alpha_1, \dots, \alpha_s) : \alpha_i \ge 0 \text{ and }\sum_i \alpha_i = 1\}$, and
that $\Delta_{s^2}(\pi)$ denotes the set of $(\alpha_{11} \dots, \alpha_{ss})$ such that
\begin{align*}
  \alpha_{ij} & \ge 0 \text{ for all $i, j$,} \\
  \sum_{i=1}^s \alpha_{ij} &= \pi_j \text{ for all $j$, and} \\
  \sum_{j=1}^s \alpha_{ij} &= \pi_i \text{ for all $i$.}
\end{align*}

In what follows, we fix an $s^2 \times s^2$ matrix $A$ and some $\pi \in \Delta_s$.
We define $p \in \Delta_{s^2}(\pi)$ by $p_{ij} = \pi_i \pi_j$ (or alternatively, $p = \pi^{\otimes 2}$),
and we take $N \sim \Multinom(n, p)$ and $X = (N - np)/\sqrt{n}$. Finally, fix a sequence $a_n$
such that $\sqrt n \ll a_n \ll n$ and
define $\Omega_n$ to be the event that
\begin{align}
 \max_j \left|\sum_i N_{ij} - n \pi_j \right| &\le a_n \label{eq:N_ij-1} \\
 \max_i \left|\sum_j N_{ij} - n \pi_i \right| &\le a_n. \label{eq:N_ij-2}
\end{align}
Note that the condition $\sqrt n \ll a_n$
ensures that the probability of $\Omega_n$ converges to 1.

\begin{proposition}\label{prop:UI-multinomial}
Define 
\[\lambda = \sup_{\alpha \in \Delta_{s^2}(\pi)} \frac{(\alpha - p)^T A (\alpha - p)}{D(\alpha, p)}.\]
If $\lambda < 1$ then
\[
 \E [\1_{\Omega_n} \exp(Y^T A Y)] \to \E \exp (Z^T A Z) < \infty,
\]
as $n \to \infty$, where $Z \sim \Normal(0, \diag(p) - p p^T)$.
On the other hand, if $\lambda > 1$ then
\[
 \E [\1_{\Omega_n}\exp(Y^T A Y)] \to \infty
\]
as $n \to \infty$.
\end{proposition}

\begin{lemma}\label{lem:UI-multinomial}
 For any $\epsilon > 0$, any $k = 2, 3, \dots$, and any $p \in \Delta_k$, there is a constant $C < \infty$
 such that for any $n$,
 \[
  n^{-k/2} \sum_{r_1 + \cdots + r_k = n} \exp\left(-n \epsilon\left|\frac rn - p\right|^2\right) \le C.
 \]
\end{lemma}

\begin{proof}
 We have
 \begin{align*}
  n^{-k/2} \sum_{r_1 + \cdots + r_k = n} \exp\left(-n \epsilon\left|\frac rn - p\right|^2\right)
  &\le n^{-k/2} \sum_{r_1, \dots, r_k =1}^n \exp\left(-n \epsilon\left|\frac rn - p\right|^2\right) \\
  &= \prod_{i=1}^k \left[ n^{-1/2} \sum_{r=1}^n \exp\left(-n \epsilon\left(\frac rn - p_i\right)^2\right) \right].
 \end{align*}
 The problem has now reduced to the case $k=1$; i.e., we need to show that
 \[
  n^{-1/2} \sum_{r=1}^n \exp(-n\epsilon (r/n - p)^2) < C(p, \epsilon).
 \]
 We do this by dividing the sum above into $\ell = \lceil \sqrt n\rceil$ different sums. Note that if $\frac rn \ge p$ then
 \begin{equation}\label{eq:binom-ui-1}
  \left(\frac{r + \ell}{n} - p\right)^2 = \left(\frac rn - p\right)^2 + \frac{\ell^2}{n^2} + \frac{2\ell}{n}\left(\frac rn - p\right)
  \ge \left(\frac rn - p\right)^2 + \frac 1n.
 \end{equation}
 Hence, $r \ge n p$ implies
 \[
  \exp\left(-n\epsilon\left( \frac {r + \ell}{n} - p\right)^2 \right)
  \le e^{-\epsilon} \exp\left(-n\epsilon\left( \frac {r}{n} - p\right)^2 \right).
 \]
 Stratifying the original sum into strides of length $\ell$,
 \begin{align*}
  n^{-1/2} \sum_{r = \lceil pn \rceil}^n \exp(-n\epsilon (r/n - p)^2)
  &\le n^{-1/2} \sum_{r=\lceil pn \rceil}^{\lceil pn \rceil + \ell-1} \sum_{m=0}^\infty \exp(-n\epsilon ((r + m\ell)/n - p)^2).
 \end{align*}
 Now,~\eqref{eq:binom-ui-1} implies that the inner sum may be bounded by a geometric series with initial value less than 1, and ratio $e^{-\epsilon}$. Hence,
 \[
  n^{-1/2} \sum_{r = \lceil pn \rceil}^n \exp(-n\epsilon (r/n - p)^2) \le n^{-1/2} \ell \frac{1}{1-e^{-\epsilon}},
 \]
 which is bounded. A similar argument for the case $r \le pn$ completes the proof.
\end{proof}

\begin{proof}[Proof of Proposition~\ref{prop:UI-multinomial}]
 First, recall that for any $\alpha = (\alpha_{11}, \dots, \alpha_{ss}) \in \Delta_{s^2}$, we have
 $\Pr(N = \alpha n) \asymp \exp(-nD(\alpha, p))$; this just follows from Stirling's approximation.
 Next, note that $D(\alpha, p)$ is zero only
 for $\alpha = p$, and that $D(\alpha, p)$ is strongly concave in $\alpha$. Therefore, $\lambda < 1$ implies
 that there is some $\epsilon > 0$ such that
 \[
  D(\alpha, p) \ge (1+\epsilon) (\alpha - p)^T A (\alpha - p) + \epsilon |\alpha - p|^2
 \]
 for all $\alpha \in \Delta_{s^2}(p)$. Hence, any $\alpha \in \Delta_{s^2}(p)$ satisfies
 \begin{equation}\label{eq:binom-ui-2}
  \Pr(N = \alpha n) \exp(n (1 + \epsilon)(\alpha - p)^T A (\alpha-p)) \le C \exp(-n\epsilon |\alpha - p|^2).
 \end{equation}
 Recalling the definition of $\Omega_n$, we write (with a slight abuse of notation)
 $\alpha \in \Omega_n$ if $|\max_i \sum_j \alpha_{ij} - p_i| \le n^{-1} a_n$ and similarly with $i$ and $j$ reversed.
 Note that for every $\alpha \in \Omega_n$, there is some $\tilde \alpha \in \Delta_{s^2}(\pi)$ with
 $|\alpha - \tilde \alpha|^2 = o(n^{-1})$; in particular,~\eqref{eq:binom-ui-2}
 also holds for all $\alpha \in \Omega_n$ (with a change in the constant $C$).
 Then
 \begin{align*}
  \E [\1_{\Omega_n} \exp((1 + \epsilon) X^T A X)]
  &= \sum_{\alpha \in \Omega_n} \Pr(N = n \alpha) \exp\left(n (1+\epsilon) (\alpha - p)^T A (\alpha - p)\right) \\
  &\le \sum_{\alpha \in \Omega_n} \exp\left(-n \epsilon |\alpha - p|^2\right) \\
  &\le C < \infty,
 \end{align*}
 for some constant $C$ independent of $n$, where the last line follows from Lemma~\ref{lem:UI-multinomial}.
 In particular, $\exp(X^T A X)$ has $1+\epsilon$ uniformly bounded moments, and so it is uniformly integrable
 as $n \to \infty$. Since $X \toD \normal(0, \diag(p) - p p^T)$, it follows that
 $\E \exp(X^T A X) \to \E \exp(X^T A X)$.

 In the other direction, if $\lambda > 1$ then there is some $\alpha \in \Delta_{s^2}(p)$, $\alpha \ne p$ and some
 $\epsilon > 0$ such that
 $D(\alpha, p) \le (\alpha - p)^T A (\alpha - p) - 2\epsilon$. By the continuity of $D(\alpha, p)$ and
 $(\alpha - p)^T A (\alpha - p)$, we see that for sufficiently large $n$, there exists $r \in n\Delta_{s^2}(p)$ such that
 \[
  D(r/n, p) \le (r/n - p)^T A (r/n - p) - \epsilon.
 \]
 For any $n$, let $r^* = r^*(n)$ be such an $r$. Then
 \begin{align*}
  \E \exp(X^T A X)
  &\ge \Pr(N = r^*(n)) \exp\left(n (r^*/n - p)^T A (r^*/n - p)\right) \\
  &\asymp \exp\left(n \left( (r^*/n - p)^T A (r^*/n - p) - D(r^*/n, p)\right)\right) \\
  &\ge \exp\left(n \epsilon\right) \to \infty.
 \end{align*}
\end{proof}


\end{document}